\documentclass[twoside,11pt]{amsart}

\usepackage{amsmath,latexsym,amssymb,mathptm, times}
\input amssym.def
\input amssym
\input xypic
\input xy
\xyoption{all}
\newtheorem{Theorem}{Theorem}[section]
\newtheorem{Lemma}[Theorem]{Lemma}
\newtheorem{Corollary}[Theorem]{Corollary}
\newtheorem{Proposition}[Theorem]{Proposition}
\newtheorem{Observation}[Theorem]{Observation}

\newtheorem{Question}[Theorem]{Question}

\theoremstyle{definition}
\newtheorem{Definition}[Theorem]{Definition}
\newtheorem{Remark}[Theorem]{Remark}
\newtheorem{Data}[Theorem]{Data}
\newtheorem{notation}[Theorem]{Notation}

\numberwithin{equation}{section}

\begin{document}

\baselineskip=16pt

\title[The Weak Lefschetz Property for monomial complete intersections]{\bf The Weak Lefschetz Property for monomial complete intersections in positive characteristic}

\author[Andrew R. Kustin and Adela Vraciu]
{ Andrew R. Kustin  and Adela Vraciu}

\thanks{AMS 2010 {\em Mathematics Subject Classification}.
Primary 13D02; Secondary 13A35, 13E10, 13C40.}

\thanks{Both authors were  supported in part by the NSA}


\thanks{Keywords: Artinian Rings, Complete Intersections,  Frobenius Endomorphism, Weak Lefschetz Property}

\address{Department of Mathematics, University of South Carolina,
Columbia, SC 29208} \email{kustin@math.sc.edu}

\address{Department of Mathematics, University of South Carolina,
Columbia, SC 29208} \email{vraciu@math.sc.edu}

\begin{abstract}  Let $A=\pmb k[x_1,\dots,x_n]/{(x_1^d,\dots,x_n^d)}$, where $\pmb k$ is an infinite field.
If $\pmb k$ has characteristic zero, then Stanley proved that $A$ has the  Weak Lefschetz Property (WLP).
Henceforth, $\pmb k$ has positive characteristic $p$.
If $n=3$, then Brenner and Kaid have identified all $d$, as a function of $p$,  for which $A$ has the WLP. 
In the present paper, the analogous project is carried out for $4\le n$. If $4\le n$ and $p=2$, then $A$ has the WLP if and only if $d=1$. If $n=4$ and $p$ is odd, then we prove that  $A$ has the WLP if and only if 
$d=kq+r$ for integers $k,q,d$ with $1\le k\le \frac{p-1}2$, $r\in\left\{\frac{q-1}2,\frac{q+1}2\right\}$, and $q=p^e$ for some non-negative integer $e$. If $5\le n$, then we prove that $A$ has the WLP if and only if $\left\lfloor\frac{n(d-1)+3}2\right\rfloor\le p$.
We first  interpret the WLP for the ring ${{\pmb k}[x_1, \ldots, x_{n}]}/{(x_1^d, \ldots, x_{n}^d)}$ in terms of the degrees of the non-Koszul relations on the elements $x_1^d, \ldots, x_{n-1}^d, (x_1+ \ldots +x_{n-1})^d$ in the polynomial ring $\pmb k[x_1, \ldots, x_{n-1}]$.   We then  exhibit a sufficient condition for ${{\pmb k}[x_1, \ldots, x_{n}]}/{(x_1^d, \ldots, x_{n}^d)}$ to have the WLP. This condition is expressed in terms of the non-vanishing in $\pmb k$ of determinants   of various Toeplitz matrices of binomial coefficients.  Frobenius techniques are used to produce relations of low degree on  $x_1^d$, $\ldots$, $x_{n-1}^d$, ${(x_1+ \ldots +x_{n-1})^d}$. From this we    obtain   a necessary condition for $A$ to have the WLP. We prove that the necessary condition is  sufficient by showing that the relevant determinants   are non-zero in $\pmb k$.
\end{abstract}

\maketitle

\section{Introduction}
Let $A=\bigoplus A_i$ be a standard graded algebra over the field $\pmb k$. Then $A$ has the {\it Weak Lefschetz Property} (WLP) if there exists a linear form $L$ of $A_1$ such that multiplication by $L$ from $A_i\to A_{i+1}$ has maximal rank for each index $i$. In this case, $L$ is called a {\it Lefschetz element} of $A$. The set of Lefschetz elements forms a (possibly empty) Zariski open subset of $A_1$. 

The WLP has been much investigated in recent times; see, for example, \cite{BK,CN,HSS,LZ,MMN,MMN-p}. Our interest in this topic is sparked by the following result, which is established in \cite{KRV}.

\begin{Theorem}\label{intro} Let $\pmb k$ be a field, $P$ be the polynomial ring $\pmb k[x_1,x_2,x_3]$, $n$, $N$, and $a$ be positive integers, $f=x_1^n+x_2^n+x_3^n$, $R_n=P/(f)$, and $\overline{R}_{n,N}=R_n/(x_1^N,x_2^N,x_3^N)$. Assume that $n$ does not divide $N$.   \begin{enumerate}
\item The following two statements are equivalent.
\begin{enumerate}
\item The ring $\bar {\pmb k}[x_1,x_2,x_3]/(x_1^a,x_2^a,x_3^a)$ does not have the WLP, where $\bar {\pmb k}$ is the algebraic closure of $\pmb k$.
\item There exists a non-zero relation of degree less than $\lfloor \frac {3a}2\rfloor$ on $x_1^a,x_2^a,(x_1+x_2)^a$ in $\pmb  k[x_1,x_2]$.\end{enumerate}
\item The following three statements are equivalent.
\begin{enumerate}
\item The $R_n$-module $\overline{R}_{n,N}$ has finite projective dimension.
\item The ring $\overline{R}_{n,N}$ has Cohen-Macaulay type $2$.
\item The algebra $\operatorname{Tor}_{\bullet}^P(\overline{R}_{n,N},\pmb k)$ is in the class ${\text{\bf H}}(3,2)$ of \cite{AKM,A89,A11}. \end{enumerate}
\item The assertions of {\rm(2)} hold if and only if the
the assertions of {\rm(1)} hold for at least one of the integers   $ a=\lfloor\frac Nn\rfloor$ or $a=\lceil \frac Nn\rceil$. In particular,
$$\mathrm{pd}_{R_n}\overline{R}_{n,N}<\infty\iff \bar {\pmb k}[x_1,x_2,x_3]/(x_1^a,x_2^a,x_3^a)\text{ does not have the WLP for $\textstyle a=\lfloor\frac Nn\rfloor$ or $\textstyle a=\lceil \frac Nn\rceil$}.$$
\end{enumerate}\end{Theorem}

The paper \cite{KRV}  is about the resolution of $\overline{R}$ by free $R$-modules. In particular, item (2a) is one of the main concerns in \cite{KRV}. Early in the investigation that lead to \cite{KRV}, we found a relationship between (2a) and (1b). Eventually, we found the equivalence of (1a) and (1b)  in \cite{BK} and we used the numerical values given in \cite{BK} to prove (3). Lucho Avramov drew our attention to the equivalence of (2a), (2b), and (2c) in a recent conversation. %
%
%
When we wrote \cite{KRV} we were surprised by conclusion (3); that is, we were surprised that the homological questions considered in \cite{KRV} were related to the WLP. Furthermore,  we noticed that Li and Zanello \cite{LZ} had found ``a surprising, and still combinatorially obscure, connection'' between the monomial complete intersection ideals in three variables which satisfy the WLP, as a function of the characteristic of the base field, and the enumeration of plane partitions. In the mean time, the  connection between the WLP and the enumeration of plane partitions has started to become less obscure and has started to be exploited; see \cite{CN,CHJL}. At any rate, we now make sense of, generalize, and exploit the equivalence of (1a) and (1b). 

For a complete, up-to-date, history of the WLP see \cite{MN}. In particular, the present paper focuses on the WLP for monomial complete intersections. Much is known about the WLP for rings which are not defined by monomial ideals and for rings which are not complete intersections; see \cite{MN}. Furthermore,
J. Watanabe \cite[pg. 3165, Rmk. (3)]{W98} knew some version of the equivalence of (1a) and (1b) from Theorem \ref{intro} in 1998 and this idea also is used in \cite{HMNW}.

Let $A=\pmb k[x_1,\dots,x_n]/{(x_1^d,\dots,x_n^d)}$, where $\pmb k$ is an infinite field.
If $\pmb k$ has characteristic zero, then Stanley \cite{St80} (see also \cite{W,RRR}) proved that $A$ has the  Weak Lefschetz Property (WLP). {\bf The story is much different in positive characteristic!}
Henceforth, $\pmb k$ has positive characteristic $p$.
If $n=3$, then Brenner and Kaid \cite[Cor.~2.2 and Thm.~2.6]{BK} have identified all $d$, as a function of $p$,  for which $A$ has the WLP. (Our version of the Brenner-Kaid answer may be found as \cite[Thm. 5.11]{KRV}.) In the present paper, the analogous project is carried out for $4\le n$. 
If $4\le n$ and $p=2$, then $A$ has the WLP if and only if $d=1$; see Remark \ref{char2}.  If $n=4$ and $p$ is odd, then we prove in Theorem \ref{4} that  $A$ has the WLP if and only if 
$d=kq+r$ for integers $k,q,d$ with $1\le k\le \frac{p-1}2$, $r\in\left\{\frac{q-1}2,\frac{q+1}2\right\}$, and $q=p^e$ for some non-negative integer $e$. If $5\le n$, then we prove in Theorem \ref{last} that $A$ has the WLP if and only if $\left\lfloor\frac{n(d-1)+3}2\right\rfloor\le p$.

Basically there are five ingredients to our proof.

\smallskip\noindent {\bf (1)} We use ideas that we learned from \cite{MMN}   to interpret the WLP for the ring ${{\pmb k}[x_1, \ldots, x_{n}]}/{(x_1^d, \ldots, x_{n}^d)}$ in terms of the degrees of the non-Koszul relations on the elements $x_1^d, \ldots, x_{n-1}^d, (x_1+ \ldots +x_{n-1})^d$ in the polynomial ring $\pmb k[x_1, \ldots, x_{n-1}]$. 
In particular, we recover the result  that conditions (1a) and (1b) from Theorem~\ref{intro} are equivalent.  (As previously noted, we learned about this equivalence from \cite{BK}; but it was also known by \cite{W98} and \cite{HMNW}.) This step is carried out in Section 2. 

\smallskip\noindent {\bf (2)} We obtain sufficient conditions that guarantee that  the ring ${{\pmb k}[x_1, \ldots, x_{n}]}/{(x_1^d, \ldots, x_{n}^d)}$  has the WLP. These conditions are based on estimates of the minimal generating degree of
 $$\frac{(x_1^d,\dots,x_{n-1}^d)\!:\!(x_1+\dots+x_{n-1})^\gamma}{(x_1^d,\dots,x_{n-1}^d)}$$  \noindent for various choices of $\gamma$ (not only $\gamma=d$ as ingredient (1) might suggest). This minimal generator degree is known explicitly by Reid, Roberts, and Roitman \cite{RRR} (and implicitly by Stanley \cite{St80}) if the characteristic of $\pmb k$ is zero; our calculations take place when the field has positive characteristic.    This step is carried out in Section 3. 

\smallskip\noindent {\bf (3)}  
Theorems~\ref{n=4} and \ref{5<=n}  exhibit  sufficient conditions for $A=\pmb k[x_1,\dots,x_n]/{(x_1^d,\dots,x_n^d)}$ to have the WLP when $n=4$ and $5\le n$, respectively. These conditions are expressed in terms of the non-vanishing in $\pmb k$ of determinants  of various matrices  ``$M_{d,c,c,c}$'' of binomial coefficients. These determinants have been calculated classically; see \cite{R}.

\smallskip\noindent {\bf (4)}  We use Frobenius techniques to find relations of low degree on $$x_1^d,\dots, x_{n-1}^d, {(x_1+\dots +x_{n-1})^d}.$$ This calculation produces our necessary conditions on $d,p,n$ for $A$ to have the WLP. Section 4 culminates in Theorem~\ref{3} with a necessary condition for $A$ to have the WLP, when $n=4$. The corresponding result for $5\le n$ is Theorem \ref{6.3}. 

\smallskip\noindent {\bf (5)} We prove that the necessary condition of (4) is indeed sufficient by showing   the relevant determinants ``$\det M_{d,c,c,c}$''     are non-zero in $\pmb k$. 
Lemma~\ref{.Delta} treats the case $n=4$; the case $5\le n$ is contained in the proof of Theorem \ref{5<=n}.
The verification that the relevant $\det M_{d,c,c,c}$ are non-zero is much easier when $5\le n$ than when $n=4$; because when $5\le n$, then every integer that appears in the classical expression for the factorization of $\det M_{d,c,c,c}$ (see Proposition \ref{32-29-p69}) is less than $p$; whereas, when $n=4$, one must actually count the number of $p$'s that appear in the factorization for $\det M_{d,c,c,c}$.

\bigskip

We  use the convention that if $S$ is a statement, then  
\begin{equation}\label{chi}\chi(S)=\begin{cases} 1&\text{if $S$ is true}\\0&\text{if $S$ is false}.\end{cases}\end{equation}For example, if $n$ is an integer, then $\lfloor\frac{n}2\rfloor=\frac{n-\chi(\text{$n$ is odd})}2$.

If $m$ is a homogeneous element of a graded module $M=\bigoplus_{i\in \mathbb Z} M_i$, then we write $\deg m$ for the degree of  $M$. We use $\pmb s^n(\underline{\phantom{X}})$ to indicate that the degree of an element has been shifted by $n$. In other words, if $m$ is an element of the graded module $M$, and $n$ is an integer, then $\pmb s^{n}(m)$ is the element of $M(-n)$ which corresponds to $m$. In particular, $$\deg \pmb s^n(m)=\deg m+n.$$So, \begin{equation}\label{0.5}m\in M_{\deg m}\implies \pmb s^n(m)\in M(-n)_{\deg m+n}.\end{equation}

\begin{Definition} If $M=\bigoplus_{i\in \mathbb Z} M_i$ is a  graded module, then $i_0\le \operatorname{mgd}\, M$ means $M_i=0$ for all $i<i_0$ and $i_0= \operatorname{mgd}\, M$ means $M_i=0$ for all $i<i_0$ with $M_{i_0}\neq 0$. In particular, if $M$ is the zero module, then $\infty=\operatorname{mgd}\, M$ and if $M$ is a finitely generated non-zero graded module, then $\operatorname{mgd}\, M$ is an integer. The abbreviation $\operatorname{mgd}\,$ stands for  {\it minimal generator degree}.  \end{Definition}

\begin{Definition}\label{xi}
Fix the data $(\pmb k,n,\pmb a)$, where  $\pmb k$ is a field,  $n$ is a positive integer, and $\pmb a=(a_1,\dots,a_n)$  is an ordered $n$-tuple of non-negative integers. Define   $\xi{(\pmb k,n,\pmb a)}$  to be the
 homogeneous $\pmb k[x_1,\dots,x_{n-1}]$-module map $$\bigoplus_{i=1}^n \pmb k[x_1,\dots,x_{n-1}](-a_i)\to \pmb k[x_1,\dots,x_{n-1}],$$ which is given by the matrix
$$[x_1^{a_1},\dots,   x_{n-1}^{a_{n-1}},(x_1+\dots+x_{n-1})^{a_n}],$$$\mathrm {Syz}(\pmb k,n,\pmb a)$  the kernel of $\xi{(\pmb k,n,\pmb a)}$, $\mathrm {Kos}(\pmb k,n,\pmb a)$   the $\pmb k[x_1,\dots,x_{n-1}]$-submodule of 
$\mathrm {Syz}(\pmb k,n,\pmb a)$ which is generated by the Koszul relations on 
$\{x_1^{a_1},\dots,   x_{n-1}^{a_{n-1}},(x_1+\dots+x_{n-1})^{a_n}\}$,  $$\overline{\mathrm {Syz}}(\pmb k,n,\pmb a)=\frac{\mathrm {Syz}(\pmb k,n,\pmb a)}{\mathrm {Kos}(\pmb k,n,\pmb a)}, $$  $A(\pmb k,n,\pmb a)$  the quotient $$A(\pmb k,n,\pmb a)=\frac{\pmb k[x_1,\dots,x_n]}{(x_1^{a_1},\dots,x_n^{a_n})},$$
$K(\pmb k,n,\pmb a)$  the kernel  of the homogeneous $\pmb k[x_1,\dots,x_n]$-module homomorphism  $$\xymatrix{A(\pmb k,n,\pmb a)(-1)\ar[rr]^{\phantom{xx}L(\pmb k,n,\pmb a)}&& A(\pmb k,n,\pmb a)},$$where $L(\pmb k,n,\pmb a)$ is the  linear form $x_1+...+x_n$ of $\pmb k[x_1,\dots,x_n]$,
and $J(\pmb k,n,\pmb a,\gamma)$ to be the ideal 
 $$J(\pmb k,n,\pmb a,\gamma)=\frac{(x_1^{a_1},\dots,x_n^{a_n})\!:\!(x_1+\dots +x_n)^\gamma}{(x_1^{a_1},\dots,x_n^{a_n})}$$of $A(\pmb k,n,\pmb a)$. 
\end{Definition}

\begin{Remark}
We use the notation $a_1\!:\!r$ to mean that the integer $a_1$ appears $r$ times. So, in particular, if $d$ is a non-negative integer, then  the map
$\xi(\pmb k,n,d\!:\!n)$ is represented by the matrix $$[x_1^d,x_2^d,\dots,x_{n-1}^d,(x_1+\dots+x_{n-1})^d],$$ $A(\pmb k,n,d\!:\!n)$ is the quotient  $$A(\pmb k,n,d\!:\!n)=\frac{\pmb k[x_1,\dots,x_n]}{(x_1^{d},\dots,x_n^{d})},$$
and if $k$ and $\ell$ are non-negative integers with $\ell\le n$, then the map $\xi(\pmb k,n,(k+1)\!:\!\ell,k\!:\!(n-\ell))$  is represented by the matrix 
$$[x_1^{k+1} ,\dots, x_\ell^{k+1},x_{\ell+1}^k,\dots, x_{n-1}^{k},(x_1+\dots+x_{n-1})^k].$$ 
\end{Remark}

\begin{Remark} For data $(\pmb k,n,\pmb a)$ as described in Definition \ref{xi}, $\mathrm {Kos}(\pmb k,n,\pmb a)$ is the submodule of $\mathrm {Syz}(\pmb k,n,\pmb a)$ which is generated by all relations of the form:
$$\left[\begin{matrix} 0,\cdots ,0,g_j,0,\cdots ,0,-g_i,0,\cdots ,0
\end{matrix}\right]^{\rm t}$$ where $g_j$ appears in row $i$, $-g_i$ appears in row $j$ and $[g_1,\dots,g_n]=[x_1^{a_1},\dots x_{n-1}^{a_{n-1}},(x_1+\dots+x_{n-1})^{a_n}]$. 
\end{Remark}

\begin{Data}\label{xi'} Fix the data $(\pmb k,n,\pmb a)$,  where  $\pmb k$ is a field,  $n$ is a positive integer, and $\pmb a=(a_1,\dots,a_n)$  is an ordered $n$-tuple of positive integers. \end{Data}


\begin{notation}\label{E}If $n$ is a positive integer, $\pmb a=(a_1,\dots,a_n)$ is an $n$-tuple of integers and $\gamma$ is an integer, then let $|\pmb a|$, $E(n,\pmb a)$,  and $\mathrm{MN}(n,\pmb a,\gamma)$ represent the integers $$|\pmb a|=\sum_{i=1}^na_i,\qquad E(n,\pmb a)=\left\lfloor \frac{|\pmb a|-n+3}2\right\rfloor,\qquad\text{and}\qquad \mathrm{MN}(n,\pmb a,\gamma)= 1+\left\lfloor \frac{|\pmb a|-n -\gamma}2\right\rfloor.$$
\end{notation} 

Many of our results are stated in terms of the relationship between the integers $E(n,\pmb a)$ and $\operatorname{mgd}\, \overline{\mathrm {Syz}}(\pmb k, n, \pmb a)$ or between the integers $\mathrm{MN}(n,\pmb a,\gamma)$ and $\operatorname{mgd}\,J(\pmb k,n,\pmb a,\gamma)$.
The   connection between these relationships and the WLP for $A(\pmb k, n, \pmb a)$ is explained in Corollary \ref{degree}. 

\begin{Remark}If $A$ is a graded Artinian Gorenstein ring, then we write $\operatorname{socdeg}(A)$ for the socle degree of $A$. If $\sigma=\operatorname{socdeg}(A)$, then $A_\sigma\neq 0$, but $A_i=0$ for all $i$ with $\sigma <i$. In particular, in the language of Data \ref{xi'}, \begin{equation}\label{socdeg}\operatorname{socdeg} (A(\pmb k,n,\pmb a))=|\pmb a|-n.\end{equation} Indeed, the monomial $x_1^{a_1-1}\cdots x_n^{a_n-1}$ of $\pmb k[x_1,\dots, x_n]$ represents a basis element of the socle of $A(\pmb k,n,\pmb a)$, which is a one-dimensional vector space.\end{Remark}

Finally, we observe that our techniques also apply to the ring $\pmb k[x_1,\dots,x_n]/(x_1^{a_1},\dots,x_n^{a_n})$, even when the $a_i$'s do not all take the same value. Indeed, when $a_1\le a_2\le a_3\le a_1+a_2$ and $\pmb k$ is an infinite field, then the question ``Does $\pmb k[x_1,x_2,x_3]/(x_1^{a_1},x_2^{a_2},x_3^{a_3})$ have the WLP?'' is equivalent to 
\begin{Question}\label {han} Is the syzygy module for
$$\xymatrix{\pmb k[x_1,x_2](-a_1)\oplus \pmb k[x_1,x_2](-a_2)\oplus \pmb k[x_1,x_2](-a_3)\ar[rrr]^{\phantom{xxxxxxxxxxxxxxxxxxxxxxxx}[x_1^{a_1}\,,\,x_2^{a_2}\,,\,(x_1+x_2)^{a_3}]}&&&
\pmb k[x_1,x_2]}$$ isomorphic to $$\pmb k[x_1,x_2](-b_1)\oplus \pmb k[x_1,x_2](-b_2),$$ with $b_1=\lfloor\frac{a_1+a_2+a_3}2\rfloor$ and $b_2=\lceil\frac{a_1+a_2+a_3}2\rceil${\rm?}\end{Question}
\noindent Question \ref{han} is completely answered in Han's thesis \cite{han} for all data $(a_1,a_2,a_3,p)$, where $p$ is the characteristic of $\pmb k$. Our techniques  reproduce Han's answer to Question \ref{han}.

\section{The WLP and degrees of relations}

Retain the notation of Data {\rm \ref{xi'}}. In Corollary \ref{degree} we  translate the weak Lefschetz property for $A(\pmb k,n,\pmb a)$ into a condition on the minimal generator degree of $\overline{\mathrm {Syz}}(\pmb k,n,\pmb a)$. 
 In particular, we recover  the equivalence of (1a) and (1b) from Theorem~\ref{intro} when $n=3$. The modules $K(\pmb k,n,\pmb a)$ and $\overline{\mathrm {Syz}}(\pmb k,n,\pmb a)$ may be found in Definition \ref{xi}.

\begin{Theorem}\label{Translation} Fix   $(\pmb k,n,\pmb a)$ as in Data {\rm \ref{xi'}}. Then the  graded $\pmb k[x_1,\dots,x_n]$-modules $K(\pmb k,n,\pmb a)$ and $\overline{\mathrm {Syz}}(\pmb k,n,\pmb a)$ are isomorphic.\end{Theorem}

\begin{proof}We abbreviate $A(\pmb k,n,\pmb a)$, $K(\pmb k,n,\pmb a)$, $\mathrm {Syz}(\pmb k,n,\pmb a)$, and $\overline{\mathrm {Syz}}(\pmb k,n,\pmb a)$ as   $A$, $K$, $\mathrm {Syz}$, and $\overline{\mathrm {Syz}}$, respectively.  Let $P$ and $Q$ be the polynomial rings $P=\pmb k[x_1,\dots,x_n]$ and $Q=\pmb k[x_1,\dots,x_{n-1}]$, and let 
$L$ be the  linear form $x_1+...+x_n$ of $P$. View $Q$ as a subalgebra of $P$ and also view $Q$ as the quotient of $P$ under the $Q$-algebra surjection $\varphi:P\to Q$ with $\varphi(x_n)=-(x_1+...+x_{n-1})$. Notice that the ideal $(L)$ of $P$ is the kernel of $\varphi$. 
The ring $Q$ is the homomorphic image of the ring $P$ under $\varphi$; so, every $Q$-module is also a $P$-module.
In particular, $\overline{\mathrm {Syz}}$ is a graded $P$-module. 

We first define a homogeneous $P$-module homomorphism  $\alpha:K\to \overline{\mathrm {Syz}}$.
Let $B$ be a homogeneous element of $(x_1^{a_1},\dots,x_n^{a_n}):_PL$. It follows that there exist homogeneous polynomials  $B_1,\dots, B_n$ in $P$ with  $BL=\sum_{i=1}^nB_ix_i^{a_i}$ and $\deg B_ix_i^{a_i}=\deg BL$ for all $i$. Let $\pmb b$ be the element $$\pmb b=[\varphi(B_1),\dots,\varphi(B_{n-1}),(-1)^{a_n}\varphi(B_{n})]^{\rm t}\text{ of }\bigoplus_{i=1}^nQ(-a_i).$$ It is clear that $\pmb b$ is in $\mathrm {Syz}$ because when  $\varphi$ is applied to 
$BL=\sum_{i=1}^nB_ix_i^{a_i}$, one obtains 
$$0=\sum_{i=1}^{n-1}\varphi(B_i)x_i^{a_i}+\varphi(B_n)(-(x_1+\dots+x_{n-1}))^{a_n}=\left[x_1^{a_1},\ ...,\ x_{n-1}^{a_{n-1}},\ (x_1+...+x_{n-1})^{a_{n}}\right]\pmb b.$$
Ultimately, $\alpha$ will send \begin{equation}\label{alpha}\text{the class of   $\pmb s(B)$ in $K$ to the class of $\pmb b$ in $\overline{\mathrm {Syz}}$.}\end{equation} (The shift operator $\pmb s$ is described in (\ref{0.5}).)
We need to show that this proposed map is independent of the various choices which have been made.
Notice that if $\sum_{i=1}^nB_ix_i^{a_i}= \sum_{i=1}^nB_i'x_i^{a_i}$ for some homogeneous forms $\{B_i'\}$  in $P$, with $\deg B_i$ equal to $\deg B_i'$, then $[B_1,\dots,B_n]^{\rm t}-[B_1',\dots,B_n']^{\rm t}$ is in the submodule of $\bigoplus_{i=1}^nP(-a_i)$ which is generated by the Koszul relations on $x_1^{a_1},\dots,x_n^{a_n}$, and  these Koszul relations are carried to zero in $\overline{\mathrm {Syz}}$. Observe also that  if $B$ is in the ideal $(x_1^{a_1},\dots,x_n^{a_n})$ of $P$, then the proposed map sends $B$ to  zero. We have shown that $\alpha:K\to \overline{\mathrm {Syz}}$, as described in (\ref{alpha}), is a well-defined  homomorphism of graded $P$-modules.  

Now we define a $Q$-module homomorphism  $\beta:\overline{\mathrm {Syz}}\to K$. Let $\pmb b=[B_1,\dots,B_n]^{\rm t}$ be a homogeneous element of $\mathrm {Syz}$. It follows that \begin{equation}\label{div}B_1x_1^{a_1}+\dots+B_{n-1}x_{n-1}^{a_{n-1}}+B_n(x_1+...+x_{n-1})^{a_{n}}=0 \text{ in } Q.\end{equation}
Let $B$ be the polynomial $B=B_1x_1^{a_1}+\dots+B_{n-1}x_{n-1}^{a_{n-1}}+B_n(-x_n)^{a_{n}}$ in $P$. We see that $\varphi(B)$ is equal to the left hand side of (\ref{div}); therefore, $\varphi(B)=0$ and  $B$ is divisible by $L$ in $P$. It is clear that $L(\frac BL)=B$ is in the ideal $(x_1^{a_1},\dots,x_n^{a_n})$ of $P$ and therefore the image of $\pmb s(\frac BL)$ in $A(-1)$ is in $K$. 
Ultimately, $\beta$ will send \begin{equation}\label{beta}\text{the class of   $\pmb b$ in $\mathrm {Syz}$ to the class of 
$\textstyle \pmb s(\frac BL)$ in K.} \end{equation}We need to show that this proposed map is independent of the various choices which have been made. If $\pmb b$ had been in $\mathrm {Kos}$, then it is easy to see that $\frac BL$ is in $(x_1^{a_1},\dots,x_n^{a_n})P$ and hence $\pmb s(\frac BL)$ represents the zero element in $K$. Thus, $\beta:\overline{\mathrm {Syz}}\to K$, as described in (\ref{beta}), is a well-defined $Q$-module homomorphism. We notice that $\beta$ is also a homomorphism  of $P$-modules because every element of $K$ is annihilated by $L$; so $x_n\theta+(x_1+\dots+x_{n-1})\theta=0$ for all $\theta$ in $K$. 

We show that $\beta\circ \alpha$ is the identity map on $K$. Let $B$ be a homogeneous polynomial in $P$ with $LB=\sum_{i=1}^n B_i x_i^{a_i}$ for homogeneous polynomials $B_i$ in $P$ with $\deg LB=\deg B_ix_i^{a_i}$ for all $i$. We have seen that $\sum_{i=1}^n \varphi( B_i)x_i^{a_i}$ is also divisible by $L$ in $P$. Let $B'$ be the homogeneous polynomial in $P$ with $LB'=\sum_{i=1}^n \varphi(B_i)x_i^{a_i}$. We also have seen that $\beta\circ \alpha$ takes the class of $\pmb s B$ in $K$ to the class of $\pmb s B'$ in $K$. For each $i$, we notice that $B_i-\varphi(B_i)$ is in the kernel of $\varphi$; hence 
$B_i-\varphi(B_i)$ is divisible by $L$ in $P$. It follows that $L(B-B')$ is in the ideal $L(x_1^{a_1},\dots,x_n^{a_n})$ in the domain $P$ and therefore $B-B'\in (x_1^{a_1},\dots,x_n^{a_n})$ and $B$ and $B'$ represent the same element of $A$. 

Finally, we consider the composition $\alpha\circ \beta:\overline{\mathrm {Syz}}\to \overline{\mathrm {Syz}}$. If $\pmb b=[B_1,\dots,B_n]^{\rm t}$ is a homogeneous element of $\mathrm {Syz}$, then $\beta$ takes the class of $\pmb b$ in $\overline{\mathrm {Syz}}$ to the class of 
$$\pmb s((B_1x_1^{a_1}+\dots+B_{n-1}x_{n-1}^{a_{n-1}}+B_n(-x_n)^{a_{n}})/L)$$ in $K$ and $\alpha\circ \beta$ takes the class of $\pmb b$ in $\overline{\mathrm {Syz}}$  to the class of $[\varphi(B_1),\dots,\varphi(B_{n-1}),(-1)^{a_n}(-1)^{a_n}\varphi(B_n)]^{\rm t}$ in $\overline{\mathrm {Syz}}$;  $\varphi$ acts like the identity on $Q$ and each $B_i$ is in $Q$; so,  $\alpha\circ \beta$ is the identity map on $\overline{\mathrm {Syz}}$.
\end{proof}

Recall the integers $\mathrm{MN}(n,\pmb a,\gamma)$ and $E(n,\pmb a)$ from Notation \ref{E}. Corollary \ref{degree} is a list of equivalent conditions. Most of the equivalences are either due to \cite{MMN} or are due to bookkeeping. The new part of this result is the equivalence between  (4) or (5) and any of the other conditions. We use all of the conditions somewhere in the paper. It is convenient  to have them all in one place. 

\begin{Corollary}\label{degree}Fix   $(\pmb k,n,\pmb a)$ as in Data {\rm\ref{xi'}}. Let $A$ represent $A(\pmb k,n,\pmb a)$ and $L$ represent $x_1+\dots +x_n$. The following statements are equivalent.
\begin{itemize}
\item[{\rm (1)}]The map, multiplication by $L$, from $A_i$ to $A_{i+1}$, is injective for all $i\le \left\lfloor \frac{|\pmb a|-n-1}2\right\rfloor$.
\item[{\rm (2)}]The map, multiplication by $L$, from $A_i$ to $A_{i+1}$, is injective for all $i= \left\lfloor \frac{|\pmb a|-n-1}2\right\rfloor$.
\item[{\rm (3)}]$\mathrm{MN}(n,\pmb a,1)\le \operatorname{mgd}\, J(\pmb k,n,\pmb a,1)$
\item[{\rm (4)}]$E(n,\pmb a)\le \operatorname{mgd}\, \overline{\mathrm {Syz}}(\pmb k,n,\pmb a)$
\item [{\rm (5)}]$\mathrm{MN}(n-1,(a_1,\dots,a_{n-1}),a_n)\le \operatorname{mgd}\, J(\pmb k,n-1,(a_1,\dots,a_{n-1}),a_n)$
\end{itemize}
Furthermore, if the field $\pmb k$ is infinite, then the above statements are also equivalent to 
\begin{itemize}
\item[{\rm (6)}]The ring $A$ has the WLP. 
\end{itemize}
\end{Corollary}

\begin{Remark} \label{after-degree}Often, when one applies Corollary \ref{degree}, one knows   ahead of time that $$E(n,\pmb a) \le \operatorname{mgd}\, \mathrm {Kos} (\pmb k, n,\pmb a).$$ Neither of these numbers require any algebraic calculation; in particular, $\operatorname{mgd}\, \mathrm {Kos}(\pmb k, n,\pmb a)$ is the sum of the two smallest elements of $\pmb a$.   If $E(n,\pmb a) \le \operatorname{mgd}\, \mathrm {Kos} (\pmb k, n,\pmb a)$, then $\mathrm {Syz}_i=\overline{\mathrm {Syz}}_i$ for all $i<E(n,\pmb a)$; thus, 
$$E(n,\pmb a)\le \operatorname{mgd}\,\overline{\mathrm {Syz}}(\pmb k,n,\pmb a)\iff E(n,\pmb a)\le \operatorname{mgd}\,\mathrm {Syz}(\pmb k,n,\pmb a).$$\end{Remark}

\begin{proof} It is obvious that $(1) \Rightarrow (2)$. The implication $(2)\Rightarrow (1)$ may be found in \cite[Prop. 2.1]{MMN}. The equivalence of (1) and (3) follows from the definition of $\mathrm{MN}(n,\pmb a,1)$ and $J(\pmb k,n,\pmb a,1)$. Observe that
$$\begin{array}{rcl}\text{assertion (1) holds} &\iff&
\left[\operatorname{Ker} \left(A(-1)\stackrel{L}{\longrightarrow} A\right)\right]_i=0\ \forall i\  \le \left\lfloor \frac{|\pmb a|-n+1}2\right\rfloor\vspace{5pt}\\
&\iff& \left\lfloor\frac{|\pmb a|-n+3}2\right\rfloor \le  \operatorname{mgd}\, \left[\operatorname{Ker} \left(A(-1)\stackrel{L}{\longrightarrow} A\right)\right]\vspace{5pt}\\
&\iff& 
E(n,\pmb a)\le \operatorname{mgd}\, K(\pmb k,n,\pmb a)\vspace{5pt}\\
 &\iff& \text{assertion (4) holds}
,  
\end{array}$$ where the final equivalence is due to   Theorem \ref{Translation}. The map $$\left[\begin{matrix} f_1,\cdots,f_n\end{matrix}\right]^{\mathrm t} \mapsto f_n,$$from 
$$\mathrm {Syz}(\pmb k,n,\pmb a)=\ker[x_1^{a_1},\dots,x_{n-1}^{a_{n-1}},(x_1+\dots+x_{n-1})^{a_n}]\text{ to } 
\left((x_1^{a_1},\dots,x_{n-1}^{a_{n-1}}): (x_1+\dots+x_{n-1})^{a_n}\right)(-a_n)$$
induces a degree preserving  isomorphism
$${\textstyle \overline{\mathrm {Syz}}(\pmb k,n,\pmb a)=\frac{\mathrm {Syz}(\pmb k,n,\pmb a)}{\mathrm {Kos}(\pmb k,n,\pmb a)}\to \left(\frac{(x_1^{a_1},\dots,x_{n-1}^{a_{n-1}}): (x_1+\dots+x_{n-1})^{a_n}}{(x_1^{a_1},\dots,x_{n-1}^{a_{n-1}})}\right)(-a_n)
=J(\pmb k,n-1,(a_1,\dots,a_{n-1}),a_n)(-a_n)},$$and this isomorphism explains $(4) \iff (5)$. Now assume that $\pmb k$ is infinite. 
It is shown in Propositions 2.2 and  2.1 of \cite{MMN} (see also Remark~\ref{evensocle} of the present paper) that 
$A$ has the WLP if and only if $L$ is a Lefschetz element for $A$, and this is equivalent to the assertion  that the multiplication by $L$ map from $A_{\lfloor \frac{\sigma-1}{2}\rfloor}$ to $A_{\lfloor \frac{\sigma+1}{2} \rfloor} $ is injective, where
$\sigma=|\pmb a|-n$ is the socle degree of $A$, see (\ref{socdeg}). Thus, \cite{MMN} shows $(1)\iff(6)$. 
\end{proof}
Remark \ref{evensocle} is well-known; we include a proof of it for the sake of completeness. This remark allows us to appeal to the results of \cite{MMN} as they are written.
\begin{Remark}\label{evensocle}Let $A$ be a standard graded Artinian Gorenstein algebra with even socle degree $2s$. If $L$ is an element of $A_1$ and multiplication by $L$ gives an injective map $A_{s-1}\to A_{s}$, then multiplication by $L$ gives a surjection $A_{s}\to A_{s+1}$.\end{Remark}

\begin{proof} Fix a non-zero socle element $\omega$ of $A_{2s}$. Let $a$ be a non-zero element of $A_{s+1}$. Multiplication gives a perfect pairing $A_{s-1}\times A_{s+1}\to A_{2s}$; hence, there exists a basis $b_1,\dots,b_r$ for $A_{s-1}$ such that $ab_1=\omega$ and $ab_i=0$ for $i\ge 2$. The hypothesis ensures that $Lb_1,\dots,Lb_r$ are linearly independent elements of $A_{s}$. Multiplication $A_{s}\times A_{s}\to A_{2s}$ is a perfect pairing; so there exists an element $c$ in $A_{s}$ with $c(Lb_1)=\omega$ and $c(Lb_i)=0$ for $i\ge 2$. We see that $a-Lc$ is an element of $A_{s+1}$ with $(a-Lc)A_{s-1}=0$. The fact that $A_{s-1}\times A_{s+1}\to A_{2s}$ is a perfect pairing yields $a=Lc$. \end{proof}

\section{Conditions that guarantee that $A(k,n,d\!:n)$ has the WLP.}

\begin{Data}\label{data3}Fix the data $(\pmb k,n,\pmb a,\gamma)$, where $\pmb k$ is a field, $n$ is a positive integer, $\pmb a=(a_1,\dots,a_n)$ is an $n$-tuple of positive integers, and $\gamma$  is a non-negative integer. We say that ``inequality (\ref{Stan}) holds for the data $(\pmb k,n,\pmb a,\gamma)$'' if the inequality 
\begin{equation}\label{Stan}\mathrm{MN}(n,\pmb a,\gamma)  \le\operatorname{mgd}\, J(\pmb k,n,\pmb a,\gamma)\end{equation}
holds.

\end{Data} The connection   between the inequality (\ref{Stan})
 and the Lefschetz property is made quite clear in Corollary \ref{degree}. In particular,
if the field $\pmb k$ is infinite, then 
\begin{equation*}\begin{array}{rcl}\text{$A(\pmb k,n,\pmb a)$ has the WLP}&\iff& \text {inequality (\ref{Stan}) holds 
for the data $(\pmb k,n,\pmb a,1)$}
\\
&\iff&\text {inequality (\ref{Stan}) holds  for the data $(\pmb k, n-1,(a_1,\dots,a_{n-1}),a_n)$}.
\end{array}\end{equation*}Furthermore, the same style of argument shows that $L=x_1+\dots+x_n$ is a strong Lefschetz element for $A(\pmb k,n,\pmb a)$ if and only if inequality (\ref{Stan}) holds for the data $(\pmb k,n,\pmb a,\gamma)$, for all values of  $\gamma$. 
Stanley's original proof \cite{St80} that $L=x_1+\dots+x_n$ is a strong Lefschetz element for $A(\mathbb C,n,\pmb a)$ involved the hard Lefschetz theorem from Algebraic Geometry.  Later proofs of the statement $$\text{if the characteristic of $\pmb k$ is zero, then 
$A(\pmb k,n,\pmb a)$ has the WLP}$$ pass through the inequality (\ref{Stan}); see  \cite{W} and especially \cite[Theorem 5]{RRR}. In Lemma \ref{29.4} we prove inequality (\ref{Stan}) for the data $(\pmb k,n,d\!:\! n,\gamma)$, under the hypothesis that certain matrices $M_{t,b,s,s}$, (see Definition \ref{3.1}),  have non-zero determinant in $\pmb k$. 
Lemma \ref{29.4}, as stated, includes an inductive hypothesis, but ultimately, in the applications,   this hypothesis is replaced by the assumption that $\det M_{t,b,s,s}\neq 0$ in $\pmb k$ for certain choices of $t,b,s$. 
We apply Lemma \ref{29.4} in two situations: $n=4$ (see Theorem \ref{n=4} and also Theorem \ref{4}) and $5\le n$ with $d$ small when compared to the characteristic of $\pmb k$ (see Theorem \ref{5<=n} and also Theorem \ref{last}).

The matrices $M_{t,b,s,s}$ have become ubiquitous in the study of the WLP. We first met them in \cite{LZ} where we learned that Paul Roberts calculated their determinants in \cite{R}. Roberts gives a reference to \cite{M} from 1930. These matrices   are used to count plane partitions and other combinatorial objects; see the work of Cook and Nagel \cite{CN-1,CN}; they also are used in the calculation of Hilbert-Kunz multiplicities. 
  
\begin{Definition} \label{3.1}  Let $t,b,r,c$ be integers, with $r$ and $c$ positive. Define 
$M_{t,b,r,c}$ to be the following $r\times c$ matrix of integers
$$M_{t,b,r,c}  = \left[ \begin{array}{cccc} \binom t b  & \binom t { b-1 } & \ldots & \binom t {b-c +1} \\
\binom t {b+1}  & \binom t   b   & \ldots & \binom t {b-c+2} \\
\vdots  & \vdots &  & \vdots \\
\binom t {b+r-1} & \binom  t {b+r-2} & \ldots & \binom t  {b+r-c} \\
\end{array}\right].$$(The parameters ``$t$'' and ``$b$'' stand for the top and bottom components in the binomial coefficient in the upper left hand corner and ``$r$'' and ``$c$'' stand for the number of rows and the number of columns. When $r=c$, we are likely to use the parameter ``$s$'' for side length.)
Recall that  the 
 binomial
coefficient $\binom{m}{i}$ is defined to be
$$\binom{m}{i}=\left\{ \begin{tabular}{ll}
$\dfrac{m(m-1)\cdots(m-i+1)}{i!}$&\quad if $0<i$,\\\vspace{5pt} 
$1$&\quad if $0=i$, and\\\vspace{5pt} $0$&\quad if 
$i<0$,\end{tabular} \right.$$for all integers $i$ and $m$. In particular, $\binom mi=0$ whenever $0\le m<i$.
 \end{Definition}

\begin{Proposition}\label{32-29-p69} If $b,s,t$ are integers, with  $0\le b\le t$   and $1\le s$, then  $\det M_{t,b,s,s}$ is a non-zero integer. Furthermore, if $\pmb k$ is a field of characteristic $p$ and $t+s\le p$, then  $\det M_{t,b,s,s}$ is a non-zero element of $\pmb k$. \end{Proposition}
 \begin{proof}The determinant of the matrix $M_{t,b,s,s}$ is calculated  in \cite[page 335]{R} to be 
$$
\frac{ \binom t  b   \binom {t +1} b   \cdots \binom{ t+s-1}   b }{\binom b  b   \binom {b+1} b \cdots \binom {b+s -1} b }.$$\end{proof}

We collect a few properties of $\det M_{t,b,s,s}$. 

\begin{Proposition}\label{Properties}Let $b$, $s$, and $t$ be integers, with  $0\le b\le t-1$   and $1\le s$. Then the following statements hold{\rm:}
\begin{itemize}
\item[{\rm (1)}] $\det M_{t,b,s,s}=\det M_{b+s,b,t-b,t-b}$, and 
\item[{\rm (2)}] the matrix $M_{t,b,s,s}$ is the transpose of $M_{t,t-b,s,s}$.
\end{itemize}
\end{Proposition}
\begin{proof} For (1), apply the proof of Proposition \ref{32-29-p69} to see that 
$$\begin{array}{rclcl}\det M_{t,b,s,s}&=&\frac{ \binom t  b   \binom {t +1} b   \cdots \binom{ t+s-1}   b }{\binom b  b   \binom {b+1} b \cdots \binom {b+s -1} b }
&=&\begin{cases} \frac{\phantom{\binom b  b   \binom {b+1} b \cdots \binom {b+s -1} b}\binom{b+s}b\cdots \binom{t-1}b \binom t  b   \binom {t +1} b   \cdots \binom{ t+s-1}   b }{\binom b  b   \binom {b+1} b \cdots \binom {b+s -1} b\binom{b+s}b\cdots \binom{t-1}b\phantom{\binom t  b   \binom {t +1} b   \cdots \binom{ t+s-1}   b} }&\text{if $b+s\le t$}\vspace{8pt}\\
\frac{\phantom{\binom b  b   \binom {b+1} b \cdots \binom {t-1} b} \binom t  b   \binom {t +1} b   \cdots\binom {b+s -1} b\binom {b+s} b\cdots \binom{ t+s-1}   b }{\binom b  b   \binom {b+1} b \cdots \binom {t-1} b\binom t  b   \binom {t +1} b   \cdots\binom {b+s -1} b\phantom{b\binom {b+s} b\cdots \binom{ t+s-1}   b } }&\text{if $t\le b+s$}\end{cases}\\
&=&\frac{ \binom {b+s}  b      \cdots \binom{ t+s-1}   b }{\binom b  b     \cdots \binom {t -1} b }&=&\det M_{b+s,b,t-b,t-b}.\end{array}
$$ For (2), the entry of $M_{t,t-b,s,s}$ in position $(r,c)$ is 
$$\binom{t}{t-b+(r-1)-(c-1)}=\binom t{t-b+r-c}=\binom t{b-r+c}=\binom t{b+(c-1)-(r-1)},$$ which is the entry of 
$M_{t,b,s,s}$ in position $(c,r)$. The middle equality used the fact that $\binom ab=\binom a{b-a}$ for all integers $a$ and $b$ with $0\le a$. 
\end{proof}

Often we consider the data
\begin{equation}\label{data3'}\text{$(\pmb k,n,d,\gamma)$ where $\pmb k$ is a field, $n$ and $d$ are positive integers and $\gamma$ is a non-negative integer.}\end{equation}
From this data we create the $n$-tuple $\pmb a=d\!:\!n$. 

\begin{Observation}\label{2c}Consider   $(\pmb k,n,d,\gamma)$ as described in   {\rm(\ref{data3'})}  with $n=2$ and $\gamma$ even. If inequality {\rm(\ref{Stan})} holds for the data 
$(\pmb k,n,d\!:\!n,\gamma)$ holds, then inequality {\rm(\ref{Stan})} holds for the data $(\pmb k,n,d\!:\!n,\gamma+1)$.
\end{Observation}

\begin{proof} Let $\gamma=2c$ and $L=x_1+x_2$.  We are given that $d-c\le \operatorname{mgd}\, \left(\frac{(x_1^d,x_2^d):L^{2c}}{(x_1^d,x_2^d)}\right)$. We must show that $d-c-1\le \operatorname{mgd}\, \left(\frac{(x_1^d,x_2^d):L^{2c+1}}{(x_1^d,x_2^d)}\right)$. We have  $d-c-1<d-c\le \operatorname{mgd}\, (x_1^d,x_2^d)$; so we are also given that $d-c\le \operatorname{mgd}\, \left( (x_1^d,x_2^d)\!:\!L^{2c} \right)$ and it suffices to prove that $d-c-1\le \operatorname{mgd}\, \left( (x_1^d,x_2^d)\!:\!L^{2c+1}\right)$. Take a non-zero  homogeneous element $b$ of $(x_1^d,x_2^d)\!:\!L^{2c+1}$. It follows that $bL$ is in $(x_1^d,x_2^d)\!:\!L^{2c}$; so, by hypothesis, the degree of $bL$ is at least $d-c$; and therefore, the degree of $b$ is at least $d-c-1$. \end{proof}

\begin{Lemma}\label{29.4}  Fix   $(\pmb k,n,d,\gamma)$, as described in  {\rm(\ref{data3'})} with $2\le n$, and let  $$\delta=\min\{\mathrm{MN}(n,d\!:\!n,\gamma)-1, d-1\}.$$  If $n=2$, then assume that  $\det M_{d, \lfloor \frac \gamma 2\rfloor,\lfloor \frac \gamma 2\rfloor,\lfloor \frac \gamma 2\rfloor} \neq 0$ in $\pmb k$.
 If $3\le n$, then assume that 
\begin{enumerate}\item  $\det M_{\gamma, d-1-\delta+\ell,\delta-\ell+1,\delta-\ell+1}\neq 0$ in $\pmb k$, and  \item inequality {\rm(\ref{Stan})} holds for the data $(\pmb k,n-1,d\!:\!(n-1),\gamma-2\ell+2\delta-d+1)$ \end{enumerate}for
all $\ell$ with $0\le \ell\le \min\{\delta,\gamma\}$. 
Then inequality {\rm(\ref{Stan})} holds for the data $(\pmb k,n,d\!:\!n,\gamma)$.
\end{Lemma}
\begin{proof} It is clear that  inequality {\rm(\ref{Stan})} holds for the data $(\pmb k,n,d\!:\!n,\gamma)$  when   $n(d-1)+1\le \gamma$  and also when   $\gamma=0$ .  Henceforth, we assume that 
$1\le \gamma\le n(d-1)$. Also, if $n=2$ and the conclusion holds for even $\gamma$, then Observation   \ref{2c} shows that the conclusion holds for for odd $\gamma$. Henceforth, when $n=2$ we assume that $\gamma$ is even. We prove that 
\begin{equation}\label{32-29-reduction}\begin{array}{l}\text{every element of $ (x_1^d,\dots,x_n^d)\!:\!_{\pmb k[x_1,\dots,x_n]}(x_1+\dots+x_n)^\gamma $  of degree equal to $\left\lfloor\frac{n(d-1)-\gamma}{2}\right\rfloor$}\\\text{is in  $(x_1^d,\dots,x_n^d)$}.\end{array}\end{equation} 
Indeed, once we have shown (\ref{32-29-reduction}), then the usual trick involving socle degree   yields that every homogeneous element of 
$(x_1^d,\dots,x_n^d)\!:\!_{\pmb k[x_1,\dots,x_n]}(x_1+\dots+x_n)^\gamma $ of degree at most $\left\lfloor\frac{n(d-1)-\gamma}{2}\right\rfloor$    is already in $(x_1^d,\dots,x_n^d)$. Of course, $\left\lfloor\frac{n(d-1)-\gamma}{2}\right\rfloor=\mathrm{MN}(n,d\!:\!n,\gamma)-1$.
 
In light of the goal (\ref{32-29-reduction}),
fix a homogeneous polynomial $u\in \pmb k[x_1,\dots,x_n]$, with $$u\in (x_1^d,\dots,x_n^d)\!:\!(x_1+\dots+x_n)^\gamma$$ and  \begin{equation}\label{32-29-degu}\deg u=\left\lfloor\frac{n(d-1)-\gamma}{2}\right\rfloor.\end{equation} 
 We   show that $u\in  (x_1^d,\dots,x_n^d)$. We write $u$ as an element of $(k[x_1,\dots,x_{n-1}])[x_n]$. The part of $u$ that has degree at least $d$ in $x_n$ already is in $(x_1^d,\dots,x_n^d)$. No harm is done if we ignore this part of $u$ and merely keep those terms that have degree in $x_n$ of degree $d-1$ or less. The parameter $\delta$ satisfies   $\delta=\min\{\deg u,d-1\}$.
Write $u=\sum_{j=0}^{\delta} u_jx_n^j$ with $u_j$ homogeneous of degree $\deg u-j$ in $k[x_1,\dots,x_{n-1}]$. We   show that $u_0,\dots u_{\delta}$ are in $(x_1^d,\dots,x_{n-1}^d)$. Let $L$ be the linear form $$L=x_1+\dots+x_{n-1}$$ in $\pmb k[x_1,\dots,x_{n-1}]$.
If $0\le k\le d-1$, then  the coefficient of $x_n^k$ in 
$$ (L+x_n)^\gamma u =\left(\sum\limits_{i\in \mathbb Z}\binom{\gamma }iL^{\gamma-i}x_n^i\right)\left(\sum\limits_{j=0}^{\delta} u_jx_n^j\right)=\sum\limits_{0\le k}\left(\sum\limits_{j=0}^{\delta}\binom{\gamma}{k-j}L^{\gamma+j-k}u_j \right)x_n^k$$ is in $(x_1^d,\dots,x_{n-1}^d)$. (Recall that the binomial coefficient $\binom{\gamma }i$ is defined for all integers  $i$; furthermore, the product $\binom{\gamma }iL^{\gamma-i}$ is in $\pmb k[x_1,\dots,x_{n-1}]$ because $\binom{\gamma }i$ is zero whenever $\gamma-i$ is negative.)
It follows that 
\begin{equation}\label{32-29-26.0''!}0\le k\le d-1\implies L^{\gamma-k}\sum_{j=0}^{\delta}\binom{\gamma}{k-j}L^{j}u_j\in (x_1^d,\dots,x_{n-1}^d).\end{equation} We express (\ref{32-29-26.0''!}) as a statement about the entries of a product of matrices. Each entry in the product \begin{equation}\label{32-29-26.1''!}\left [\begin{matrix} L^\gamma &\\&\ddots&\\&&L^{\gamma-(d-1)}\end{matrix}\right] M_{\gamma,0,d,\delta+1}\left[\begin{matrix}u_0L^0\\\vdots \\u_{\delta}L^{\delta} \end{matrix}\right]\end{equation}is in $(x_1^d,\dots,x_{n-1}^d)$.
The matrix on the left of (\ref{32-29-26.1''!}) is a $d\times d$ diagonal matrix; the exponent decreases by $1$ for each step down the diagonal.   The exponent starts at $\gamma$ and decreases to $\gamma-(d-1)$. The product (\ref{32-29-26.1''!}) makes sense in $\pmb k[x_1,\dots,x_{n-1}]$ because it is merely a re-phrasing of (\ref{32-29-26.0''!}) which clearly makes sense in $\pmb k[x_1,\dots,x_{n-1}]$. On the other hand, some of the individual expressions in the matrix on the left of  (\ref{32-29-26.1''!}) might actually be rational functions rather than polynomials. This does not cause a problem because before we employ (\ref{32-29-26.1''!}) (or any of its successors -- especially (\ref{32-29-26.2''!})) we multiply  on the left by a matrix of polynomials which has the effect of clearing the denominators, see (\ref{32-29-28.7}).

We have $\delta +1$ ``unknowns'' $u_0,\dots,u_{\delta}$. We need only keep $\delta+1$ equations. In other words, we may  throw away the top $d-1-\delta$ rows of (\ref{32-29-26.1''!}). This means that we may also remove the top $d-1-\delta$ rows of $M_{\gamma,0,d,\delta+1}$ to obtain 
$M_{\gamma,d-1-\delta,\delta+1,\delta+1}$. Of course, we use 
\begin{equation}\label{32-29-BD}\left[\begin{matrix} A&0\\0&B\end{matrix}\right]
\left[\begin{matrix} C\\D\end{matrix}\right]=\left[\begin{matrix} AC\\BD\end{matrix}\right].\end{equation} At this point,  each entry of the product

\begin{equation}\label{32-29-26XX!}\left [\begin{matrix} L^{\gamma-(d-1-\delta)} &\\&\ddots&\\&&L^{\gamma-(d-1)}\end{matrix}\right] M_{\gamma,d-1-\delta,\delta+1,\delta+1}\left[\begin{matrix}u_0L^0\\\vdots \\u_{\delta}L^{\delta} \end{matrix}\right]\end{equation}is in $(x_1^d,\dots,x_{n-1}^d)$.

We prove that the $u_j$ are in $(x_1^d,\dots,x_{n-1}^d)$ by descending induction on $j$ beginning at $j=\delta$ and continuing until $j=0$. As soon as we learn that a given $u_j$ is in $(x_1^d,\dots,x_{n-1}^d)$, we create a smaller square system of equations by removing $u_j$ and the top equation. We remove the top equation because it  is the equation which is multiplied   by the highest power of $L$. In practice, we find it convenient to set up the entire family of systems of equations -- one for each parameter ``$\ell $'' -- and then quickly apply the induction. To that end, we fix $\ell $ with \begin{equation}\label{32-29-fix}\begin{cases}0\le \ell \le \min\{\deg u,d-1,\gamma\}&\text{if $3\le n$}\\ \ell=0&\text{if $2=n$.}\end{cases}\end{equation} Recall that $\delta=\min\{\deg u,d-1\}$ so \begin{equation}\label{32-29-tdelta}\ell \le\delta.\end{equation}   Delete the top $\ell $ rows and the left most $\ell $ columns of the matrix on the left of (\ref{32-29-26XX!}), the top $\ell $ rows and the right most $\ell $ columns of the middle matrix, and   the bottom $\ell $ rows of the column vector on the right. We obtain that each entry of the product
\begin{equation}\label{32-29-26.2''!}\left [\begin{matrix} L^{\gamma-(d-1-\delta)-\ell } &\\&\ddots&\\&&L^{\gamma-(d-1)}\end{matrix}\right] M_{\gamma,d-1-\delta+\ell ,\delta-\ell +1,\delta-\ell +1}\left[\begin{matrix}u_0L^0\\\vdots \\u_{\delta-\ell }L^{\delta-\ell } \end{matrix}\right]\end{equation} 
is in $(x_1^d,\dots,x_{n-1}^d,u_{\delta+1-\ell },\dots,u_{\delta})$.
 Be sure to notice that the matrix on the left of (\ref{32-29-26XX!}) is a block diagonal matrix so the idea of (\ref{32-29-BD}) applies once again and so deleting  the top $\ell $ rows and the left most $\ell $ columns of the matrix on the left and  the top $\ell $ rows of the matrix in the middle merely deletes the top $\ell $ rows from the product. Also notice that it is legal to remove the right most $\ell $ columns of the middle matrix in (\ref{32-29-26XX!}) and   the bottom $\ell $ rows of the column vector on the right because we moved this information to the other side of the equation. The middle matrix $M_{\gamma,d-1-\delta,\delta+1,\delta+1}$ of (\ref{32-29-26XX!})
 becomes  $M_{\gamma,d-1-\delta+\ell ,\delta-\ell +1,\delta-\ell +1}$ in (\ref{32-29-26.2''!}) because the new entry in position $(1,1)$ is $\binom \gamma {d-1-\delta+\ell }$ and the new matrix has $\delta-\ell +1$ rows and columns.   Multiply the $(\delta-\ell +1)\times (\delta-\ell +1)$ matrix on the left of (\ref{32-29-26.2''!}) by 
the $(\delta-\ell +1)\times (\delta-\ell +1)$ matrix
\begin{equation}\label{32-29-28.7}\left [\begin{matrix} L^{0} &\\&\ddots&\\&&L^{\delta-\ell }\end{matrix}\right]\end{equation} to get $L^{\gamma-(d-1-\delta)-\ell }$ times the identity matrix, which is a scalar matrix. The scalar matrix commutes with  $M_{\gamma,d-1-\delta-\ell ,\delta-\ell +1,\delta-\ell +1}$. 
Notice that each entry of the matrix of (\ref{32-29-28.7}) is a polynomial by (\ref{32-29-tdelta}). Furthermore, we show  that $L^{\gamma-(d-1-\delta)-\ell }$ is  also a polynomial by showing that 
\begin{equation}\label{32-29-poly} d-1+\ell \le \gamma+\delta.\end{equation} Indeed, if $n=2$, then 
\begin{equation}\label{3.12.5}d-1\le d-1+\left\lfloor\frac \gamma2\right\rfloor=\min\left\{d-1+\gamma,d-1+\left\lfloor\frac \gamma2\right\rfloor\right\}=\gamma+\delta.\end{equation} To establish (\ref{32-29-poly}) for $3\le n$, we first   
observe  that
\begin{equation} \label{32-29-gammad} \gamma\le d-1\implies d-1\le \deg u.\end{equation} Indeed, if one adds $(n-1)(d-1)-\gamma$ to both sides of $\gamma\le d-1$, then one obtains $$(n-1)(d-1)\le n(d-1)-\gamma.$$ Divide to see 
$\frac{n-1}2(d-1)\le \frac{n(d-1)-\gamma}2$. The parameter $n$ is at least $3$ by hypothesis; so $1\le \frac{n-1}2$ and  $d-1\le \frac{n(d-1)-\gamma}2$.  Furthermore, $d-1$ is an integer; so, $d-1\le \lfloor\frac{n(d-1)-\gamma}2\rfloor=\deg u$. The statement (\ref{32-29-gammad}) has been established. We return to the proof of (\ref{32-29-poly}). 
 If $d-1\le \gamma$, then (\ref{32-29-poly}) holds because $\ell \le \delta$ by (\ref{32-29-tdelta}). On the other hand, if $\gamma\le d-1$, then $d-1\le \deg u$ by (\ref{32-29-gammad}) hence, $\delta=\min\{d-1,\deg u\}=d-1$ and (\ref{32-29-poly}) still holds because $\ell \le \gamma$ by (\ref{32-29-fix}). Thus, (\ref{32-29-poly}) holds in all cases and $L^{\gamma-(d-1-\delta)-\ell }$ is   a polynomial, as was promised in the paragraph beneath (\ref{32-29-26.1''!}).

Multiply (\ref{32-29-26.2''!}) on the left by (\ref{32-29-28.7}) to see that  each entry of the product 
\begin{equation}\label{32-29-7:01}M_{\gamma,d-1-\delta+\ell ,\delta-\ell +1,\delta-\ell +1}L^{\gamma-(d-1-\delta)-\ell }\left[\begin{matrix}u_0L^0\\\vdots \\u_{\delta-\ell }L^{\delta-\ell } \end{matrix}\right]\end{equation} is in the ideal $(x_1^d,\dots,x_{n-1}^d,u_{\delta+1-\ell },\dots,u_{\delta})$ of $\pmb k[x_1,\dots,x_{n-1}]$. 

At this point, we focus on the case $n=2$. We saw at the very beginning of the proof that when $n=2$, then it suffices to prove the result for even values of $\gamma$. So we assume $\gamma$ is even. We have $\ell=0$ and $\delta=\deg u=d-1-\frac \gamma2$; so
$$\det M_{\gamma,d-1-\delta+\ell ,\delta-\ell +1,\delta-\ell +1}=\det M_{\gamma,\frac \gamma2,d-\frac \gamma2,d-\frac \gamma2}=\det M_{d,\frac \gamma2,\frac \gamma2,\frac \gamma2}.$$ The final inequality is due to Proposition \ref{32-29-26.0''!}, item (1). Thus, $M_{\gamma,d-1-\delta+\ell ,\delta-\ell +1,\delta-\ell +1}$ is an invertible matrix and (\ref{32-29-7:01}) shows that 
$$u_iL^{\frac \gamma 2+i}=u_i L^{\gamma-(d-1-\delta)-\ell +i}\in (x_1^d)$$ for $0\le i\le \delta$. On the other hand, $\deg u_i+i=\deg u=d-1-\frac \gamma2$; so, $\deg u_iL^{\frac \gamma 2+i}=d-1$; and therefore, $u_i=0$ for $0\le i\le \delta$. The proof is complete when $n=2$.

Henceforth, we assume that $n=3$. Use (\ref{32-29-tdelta}) and (\ref{32-29-poly}) to
 verify that 
$$0\le d-1-\delta+\ell \le \gamma\quad\text{and}\quad 1\le \delta-\ell +1. $$
Recall from Proposition \ref{32-29-p69} that if  $b,s,t$ are integers, with  $0\le b\le t$   and $1\le s$, then  $\det M_{t,b,s,s}$ is a non-zero integer. 
It follows from  hypothesis (1) that 
$M_{\gamma,d-1-\delta+\ell ,\delta-\ell +1,\delta-\ell +1}$ is an invertible matrix over $\pmb k$.
Multiply (\ref{32-29-7:01}) by the inverse of $M_{\gamma,d-1-\delta+\ell ,\delta-\ell +1,\delta-\ell +1}$ in order to see that
\begin{equation}\label{32-29-26.3''!}u_{\delta-\ell } \in (x_1^d,\dots,x_{n-1}^d,u_{\delta+1-\ell },\dots,u_{\delta})\!:L^{\gamma-2\ell +2\delta-d+1}.\end{equation} Apply (\ref{32-29-poly}) and (\ref{32-29-tdelta}) to see   that 
$$0\le \gamma -2\ell +2\delta-d+1.$$ According to hypothesis (2), the inequality (\ref{Stan}) holds for the data $(\pmb k,n-1,d,\gamma-2\ell+2\delta-d+1)$; and  therefore, 
\begin{equation}\label{32-29-small}\begin{array}{rcl}\deg u_{\delta-\ell }&=&\deg u-(\delta-\ell )<1+\ell -\delta+\deg u =1+\ell -\delta+\left\lfloor\frac{n(d-1)-\gamma}{2}\right\rfloor\vspace{5pt}\\&=&1+\left\lfloor\frac{(n-1)(d-1)-(\gamma-2\ell +2\delta-d+1)}{2}\right\rfloor \le  \operatorname{mgd}\,\frac{(x_1^d,\dots,x_{n-1}^d)\!:L^{\gamma-2\ell +2\delta-d+1}}{(x_1^d,\dots,x_{n-1}^d)}.\end{array}\end{equation}
Suppose that we have shown that $u_{\delta +1-\ell },\dots,u_{\delta}$ are in $(x_1^d,\dots,x_{n-1}^d)$, for some $\ell $ as described in (\ref{32-29-fix}), then (\ref{32-29-26.3''!}) gives 
$$u_{\delta-\ell }\in (x_1^d,\dots,x_{n-1}^d)\!:L^{\gamma-2\ell +2\delta-d+1};$$ but (\ref{32-29-small}) shows that $\deg u_{\delta-\ell }<\operatorname{mgd}\,\frac{(x_1^d,\dots,x_{n-1}^d)\!:L^{\gamma-2\ell +2\delta-d+1}}{(x_1^d,\dots,x_{n-1}^d)}$; so, $u_{\delta-\ell }\in (x_1^d,\dots,x_{n-1}^d)$. Thus, induction gives $u_{\delta-\ell }\in (x_1^d,\dots,x_{n-1}^d)$ for all $\ell $ described in (\ref{32-29-fix}). The maximum possible value of $\ell $ is $$\min\{\deg u,d-1,\gamma\}=\min\{\delta,\gamma\}.$$ We have shown that  
\begin{equation}\label{32-29-have}u_{\delta-\min\{\delta,\gamma\}},\dots,u_{\delta} \in (x_1^d,\dots,x_{n-1}^d).\end{equation} If $\delta\le \gamma$, then $\min\{\delta,\gamma\}=\delta$; hence $u_{0},\dots,u_{\delta}$ are in $(x_1^d,\dots,x_{n-1}^d)$ and the proof is complete.

Henceforth, we assume that $\gamma<\delta$. In particular, we have $\gamma<d-1$ (since $\delta=\min\{d-1,\deg u\}$); hence, (\ref{32-29-gammad}) shows that 
$d-1\le \deg u$; and therefore $\delta=d-1$. In this case,   (\ref{32-29-have}) shows that
\begin{equation}\label{32-29-in}u_{d-1-\gamma},\dots,u_{\delta} \in(x_1^d,\dots,x_{n-1}^d).\end{equation}
Let $k$ be a fixed parameter with $\gamma\le k\le d-1$. We know from (\ref{32-29-26.0''!}) that
\begin{equation}\label{32-29-ABC}A+B+C\in (x_1^d,\dots,x_{n-1}^d),\end{equation}
where $$\begin{array}{l}A=\sum\limits_{j=0}^{k-\gamma-1}\binom{\gamma}{k-j}L^{j+\gamma-k}u_j,\vspace{5pt}\\
B=\sum\limits_{j=k-\gamma}^{k-\gamma}\binom{\gamma}{k-j}L^{j+\gamma-k}u_j,\ \text{and}\vspace{5pt}\\
C=\sum\limits_{j=k-\gamma+1}^{d-1}\binom{\gamma}{k-j}L^{j+\gamma-k}u_j.\end{array}$$
We see that $\gamma<k-j$ in $A$; hence, $\binom{\gamma}{k-j}=0$, and $A=0$. We also see that $B=u_{k-\gamma}$ and $C$ is in $(u_{k-\gamma+1},\dots,u_{d-1})$. Induction on $k$ starting at $k=d-2$ and descending until $k=\gamma$ shows that $u_{d-2-\gamma},\dots,u_0\in (x_1^d,\dots,x_{n-1}^d,u_{d-1-\gamma},\dots,u_{\delta})$. The most recent ideal is equal to $(x_1^d,\dots,x_{n-1}^d)$ by (\ref{32-29-in}), and the proof is complete.
\end{proof}

 \begin{Theorem}  \label{n=4} Fix the data $(\pmb k,d)$ with $\pmb k$ a field and $d$ a positive integer. 
If     $\det M_{d,c,c,c} \ne 0$ in   ${\pmb k}$ for all integers $c$ with $1\le c\le d$, then the following statements hold.
\begin{enumerate}
\item The inequality {\rm(\ref{Stan})} holds for the data $(\pmb k,3,d\!:\!3,d)$.  
\item If the field $\pmb k$ is infinite, then the  ring $A(\pmb k,4,d\!:\!4)$   has the WLP. 
\end{enumerate}\end{Theorem}
 \begin{proof} According to Corollary \ref{degree} it suffices to prove (1). To that end, we apply Lemma \ref{29.4} to the data $(\pmb k,n,d,\gamma)=(\pmb k,3,d,d)$. In this case, we have $$\delta=\mathrm{MN}(3,d\!:\!3,d)-1=\left\lfloor\frac{3(d-1)-d}2\right\rfloor=d-2.$$ Once we verify that \begin{enumerate}\item  $\det M_{\gamma, d-1-\delta+\ell,\delta-\ell+1,\delta-\ell+1}\neq 0$ in $\pmb k$, and  \item inequality {\rm(\ref{Stan})} holds for the data $(\pmb k,n-1,d\!:\!(n-1),\gamma-2\ell+2\delta-d+1)$ \end{enumerate}for
all $\ell$ with $0\le \ell\le \min\{\delta,\gamma\}$,  
then  we may conclude that inequality {\rm(\ref{Stan})} holds for the data $(\pmb k,n,d\!:\!n,\gamma)$.
That is, once we verify that \begin{enumerate}\item[(a)]  $\det M_{d, 1+\ell,d-1-\ell,d-1-\ell}\neq 0$ in $\pmb k$, and  \item[(b)] inequality {\rm(\ref{Stan})} holds for the data $(\pmb k,2,d\!:\!2,2(d-2-\ell)+1)$ \end{enumerate}for
all $\ell$ with $0\le \ell\le d-2$,  
then   inequality {\rm(\ref{Stan})} holds for the data $(\pmb k,3,d\!:\!3,d)$.

Proposition \ref{Properties}, item (2), shows that $\det M_{d,\ell+1,d-1-\ell,d-1-\ell}=\det M_{d,c,c,c}$ for $c=d-1-\ell$ and this determinant is non-zero in $\pmb k$ by hypothesis; so (a) has been verified. Observation \ref{2c} shows that 
 if inequality {\rm(\ref{Stan})} holds for the data $(\pmb k,2,d\!:\!2,2(d-2-\ell))$, then inequality {\rm(\ref{Stan})} also holds for the data $(\pmb k,2,d\!:\!2,2(d-2-\ell)+1)$. It is clear that  inequality {\rm(\ref{Stan})} also holds for the data $(\pmb k,2,d\!:\!2,0)$. Thus, to establish (b), it suffices to verify that inequality {\rm(\ref{Stan})} also holds for the data $(\pmb k,2,d\!:\!2,2c)$, with $1\le c\le d-1$; therefore, according to   Lemma \ref{29.4}, it suffices to verify that $\det M_{d,c,c,c}$ is not zero in $\pmb k$ for $1\le c\le d-1$ and once again this assertion is guaranteed by our hypothesis. So (a) and (b) have been verified and the proof is complete.  
\end{proof}

\begin{Corollary}\label{Cor-big-n} Fix the data $(\pmb k,n,d,\gamma)$ as described in {\rm(\ref{data3'})} with $2\le n$.
If $\left\lfloor \frac{nd-n+2+\gamma}2\right\rfloor\le p$, then the inequality {\rm(\ref{Stan})} holds for the data $(\pmb k,n,d\!:\!n,\gamma)$.
\end{Corollary}

\begin{proof} 
The proof proceeds by induction on $n$. If $n=2$ and $d+\lfloor\frac \gamma 2\rfloor\le p$, then Proposition \ref{32-29-p69} shows that $\det M_{d,\lfloor \frac \gamma 2\rfloor,\lfloor \frac \gamma 2\rfloor,\lfloor \frac \gamma 2\rfloor}$ is not zero in $\pmb k$; hence, Lemma \ref{29.4} gives that (\ref{Stan}) holds for the data $(\pmb k,2,d\!:\!2,\gamma)$. Assume by induction that the result holds at $n-1$ and that we are given $(\pmb k,n,d,\gamma)$ with $\left\lfloor \frac{nd-n+2+\gamma}2\right\rfloor\le p$. Let 
$$\delta= \min\{\mathrm{MN}(n,d\!:\!n,\gamma)-1, d-1\}=\min\left\{\left\lfloor \frac{nd-n-\gamma}2\right\rfloor,d-1\right
\}.$$  Apply Lemma \ref{29.4}. 
Once we verify that \begin{enumerate}\item  $\det M_{\gamma, d-1-\delta+\ell,\delta-\ell+1,\delta-\ell+1}\neq 0$ in $\pmb k$, and  \item inequality {\rm(\ref{Stan})} holds for the data $(\pmb k,n-1,d\!:\!(n-1),\gamma-2\ell+2\delta-d+1)$ \end{enumerate}for
all $\ell$ with $0\le \ell\le \min\{\delta,\gamma\}$,  
then  we may conclude that inequality {\rm(\ref{Stan})} holds for the data $(\pmb k,n,d\!:\!n,\gamma)$.
Proposition \ref{32-29-p69} shows that if $\gamma+\delta+1\le p$, then hypothesis (1) holds. The induction hypothesis shows that if \begin{equation}\label{3.last}\left\lfloor \frac{[(n-1)d-(n-1)+2]+[\gamma+2\delta-d+1]}{2}\right\rfloor\le p,\end{equation}then
hypothesis (2) is satisfied.  Let $Y$ be the expression on the left side of inequality (\ref{3.last}) and $X=\max\{\gamma+\delta+1,Y\}$. We complete the proof by showing that $X\le \left\lfloor \frac{nd-n+2+\gamma}2\right\rfloor$. If $d-1\le \left\lfloor \frac{nd-n-\gamma}2\right\rfloor$, then $\delta=d-1$ and
$$\gamma+\delta+1=\gamma+d\le \gamma+\left\lfloor \frac{nd-n-\gamma}2\right\rfloor+1=\left\lfloor \frac{nd-n+2+\gamma}2\right\rfloor=Y;$$ thus, $X=Y=\left\lfloor \frac{nd-n+2+\gamma}2\right\rfloor$. On the other hand, if 
$\left\lfloor \frac{nd-n-\gamma}2\right\rfloor\le d-2$, then $\delta=\left\lfloor \frac{nd-n-\gamma}2\right\rfloor$ and
$$Y=\left\lfloor \frac{(n-2)d-n+4+\gamma}{2}\right\rfloor+\delta=
\left\lfloor \frac{nd-n-\gamma}{2}\right\rfloor-d+2+\gamma+\delta\le \gamma+\delta;$$
thus, $$X=\gamma+\delta+1=\gamma+\left\lfloor \frac{nd-n-\gamma}2\right\rfloor+1=\left\lfloor \frac{nd-n+2+\gamma}2\right\rfloor.$$
\end{proof}

\begin{Theorem}  \label{5<=n}
Fix the data $(\pmb k,n,d)$, where $\pmb k$ is a field of positive characteristic $p$, and $d$ and $n$ are positive integers with $3\le n$. Assume that $\left\lfloor\frac{n(d-1)+3}2\right\rfloor\le p$. Then the following statements hold.
\begin{enumerate}
\item The inequality {\rm (\ref{Stan})} holds for the data $(\pmb k,n-1,d\!:\!(n-1),d)$.
 \item If the field $\pmb k$ is infinite, then 
the ring $A(\pmb k,n,d\!:\!n)$ has the WLP.
 \end{enumerate}
\end{Theorem}
\begin{proof} Apply Corollary \ref{Cor-big-n} to the data $(\pmb k,n-1,d,d)$ in order to prove assertion (1). Now apply   Corollary \ref{degree} in order  to prove (2). \end{proof}

\section{Producing relations of small degrees}
Recall Definition \ref{xi} and Notation \ref{E}. In this section we use the Frobenius endomorphism to produce elements of $\mathrm {Syz}(\pmb k,4,d\!:\!4)$ of low degree. This calculation gives rise to a necessary condition for \begin{equation}\label{intro-4} E(4,d\!:\!4)\le \operatorname{mgd}\,\mathrm {Syz}(\pmb k,4,d\!:\!4).\end{equation}The necessary condition obtain in Theorem \ref{3} is shown to be sufficient in Theorem \ref{4}. Of course, when the field $\pmb k$ is infinite, then the inequality (\ref{intro-4}) is equivalent to $A(\pmb k,4,d\!:\!4)$ has the WLP, see Corollary \ref{degree}. 
\begin{Lemma}\label{Lemma25.20'}Consider the data $(\pmb k,d)$, where $\pmb k$ is a field of positive characteristic $p$, and $d$ is a positive integer. Write $d=kq+r$ for integers $k,q,r$ with  $0\le r\le q-1$ and $q=p^e$, for some positive integer $e$. Then the following statements hold.
\begin{itemize}
\item[{\rm(1)}] If $1\le k$, then $\operatorname{mgd}\, \mathrm {Syz}(\pmb k,4,d\!:\!4)\le q\operatorname{mgd}\, \mathrm {Syz}(\pmb k,4,k\!:\!4)+4r$.
\item[{\rm(2)}] If $0\le k$, then $\operatorname{mgd}\, \mathrm {Syz}(\pmb k,4,d\!:\!4)\le q\operatorname{mgd}\, \mathrm {Syz}(\pmb k,4,(k+1)\!:\!4)$. 
\item[{\rm(3)}] If $E(4,d\!:\!4)\le \operatorname{mgd}\, \mathrm {Syz}(\pmb k,4,d\!:\!4)$, then $E(4,(k+1)\!:\!4)-1\le \operatorname{mgd}\, \mathrm {Syz}(\pmb k,4,(k+1)\!:\!4)$.
\end{itemize}\end{Lemma}

\begin{proof}To prove (1), let $\eta=[v_1,v_2,v_3,v_4]^{\rm t}\in \pmb k[x_1,x_2,x_3](-k)^4$ be a homogeneous non-zero element of 
$\mathrm {Syz}(\pmb k,4,k\!:\!4)$ of degree $\operatorname{mgd}\, \mathrm {Syz}(\pmb k,4,k\!:\!4)$. This gives the equation $$\xi(\pmb k,4,k\!:\!4)\cdot \eta=0.$$Apply the Frobenius endomorphism $(\underline{\phantom{x}})^q$ and multiply by $x_1^rx_2^rx_3^r(x_1+x_2+x_3)^r$ to see that
$$\left[\begin{matrix}
v_1^q  x_2^rx_3^r(x_1+x_2+x_3)^r\\
v_2^q x_1^r x_3^r(x_1+x_2+x_3)^r\\
v_3^q x_1^rx_2^r (x_1+x_2+x_3)^r\\
v_4^q x_1^rx_2^rx_3^r 
\end{matrix}\right] $$ is a non-zero homogeneous element of $\mathrm {Syz}(\pmb k,4,d\!:\!4)$ of degree $q\operatorname{mgd}\, \mathrm {Syz}(\pmb k,4,k\!:\!4)+4r$.
 For (2),  let $\eta=[v_1,v_2,v_3,v_4]^{\rm t}\in \pmb k[x_1,x_2,x_3](-k-1)^4$ be a homogeneous  element of 
${\mathrm {Syz}(\pmb k,4,(k+1)\!:\!4)}$ of degree $\operatorname{mgd}\, \mathrm {Syz}(\pmb k,4,(k+1)\!:\!4)$, with $\eta\neq 0$. It follows that 
$$ \left[\begin{matrix} v_1^qx_1^{q-r}\vspace{2pt}\\v_2^qx_2^{q-r}\vspace{2pt}\\v_3^qx_3^{q-r}\vspace{2pt}\\v_4^q(x_1+x_2+x_3)^{q-r}\end{matrix}\right]\in \pmb k[x_1,x_2,x_3](-d)^4$$   is a non-zero element of $\mathrm {Syz}(\pmb k,4,d\!:\!4)$ of degree $q\operatorname{mgd}\, \mathrm {Syz}(\pmb k,4,(k+1)\!:\!4)$.  To prove (3),  apply (2) to see that 
$$2qk-1\le 2(qk+r)-1=2d-1\le \operatorname{mgd}\, \mathrm {Syz}(\pmb k,4,d\!:\!4)\le q\operatorname{mgd}\, \mathrm {Syz}(\pmb k,4,(k+1)\!:\!4).$$ Divide by $q$, and recall that $E(4,(k+1)\!:\!4)=2k+1$,  to obtain the conclusion. 
\end{proof}

\begin{Lemma}\label{firstint} Consider the data $(\pmb k,k)$ where $\pmb k$ is a field of characteristic $p\ge 3$ and $k$ is an  integer with $1\le k\le p-1$. 
  Then the following statements are equivalent{\rm:} 
\begin{itemize}
\item[{\rm (1)}]  $E(4,k\!:\!4)-1\le \operatorname{mgd}\, \mathrm {Syz}(\pmb k,4,k\!:\!4)$,
\item[{\rm (2)}]  $1\le k \le \frac{p+1}{2}$, and  
\item[{\rm (3)}]   $E(4,k\!:\!4)=\operatorname{mgd}\, \mathrm {Syz}(\pmb k,4,k\!:\!4)$. \end{itemize}
\end{Lemma}
\begin{proof} Recall that $E(4,k\!:\!4)=2k-1$. We show  $(1)\Rightarrow (2)$. The element $$\eta=\left[\begin{matrix}x_1^{p-k}\\x_2^{p-k}\\x_3^{p-k}\\-(x_1+x_2+x_3)^{p-k}\end{matrix}\right]\in \pmb k[x_1,x_2,x_3](-k)^4$$ is a non-zero element of $\mathrm {Syz}(\pmb k,4,k\!:\!4)$ of degree $p$. If $2k-2\le \operatorname{mgd}\, \mathrm {Syz}(\pmb k,4,k\!:\!4)$, then $2k-2\le p$; hence $k\le \frac{p+1}2$, since $p$ is odd. 
For $(2)\Rightarrow (3)$, note that $k\le \frac{p+1}2\implies \left\lfloor\frac{4(k-1)+3}2\right\rfloor\le p$; hence,
Theorem~\ref{5<=n} yields \begin{equation}\label{triv}\text{inequality (\ref{Stan}) holds for the data $(\pmb k,3,d\!:\!3,k)$}.\end{equation}  On the other hand, $$\begin{array}{rcl}\text{(\ref{triv}) holds} &\iff& \mathrm{MN}(3,k\!:\!3,k)\le \operatorname{mgd}\, J(\pmb k,3,k\!:\!3,k)\\
&\iff&E(4,k\!:\!4)\le \operatorname{mgd}\,\overline{\mathrm {Syz}}(\pmb k,4,k\!:\!4)\\
&\iff& E(4,k\!:\!4)\le\operatorname{mgd}\,\mathrm {Syz}(\pmb k,4,k\!:\!4)\end{array}$$ The first equivalence is the definition of (\ref{Stan}); the second equivalence is Corollary \ref{degree}; the third equivalence is Remark \ref{after-degree} since $E(4,k\!:\!4)=2k-1<2k=\operatorname{mgd}\,\mathrm {Kos}(\pmb k,4,k\!:\!4)$. Thus, 
 $$\textstyle k\le \frac{p+1}2\implies
2k-1\le \operatorname{mgd}\, \mathrm {Syz}(\pmb k,4,k\!:\!4).$$
To complete the proof of $(2)\Rightarrow (3)$ it suffices to exhibit a relation of degree $2k-1$.
Let $$ g=\frac{x_2^k-(-x_3)^k}{x_2+x_3} \in {\pmb k}[x_2, x_3],$$ and write $(x_1+x_2+x_3)^k=x_1^k+h(x_2+x_3)$, with $h \in {\pmb k}[x_1, x_2, x_3]$. We have
$$
(x_1+x_2+x_3)^kg=x_1^kg+h(x_2^k-(-x_3)^k),
$$
which is   a relation of the desired degree on $x_1^k, x_2^k, x_3^k, (x_1+x_2+x_3)^k$. The assertion $(3)\Rightarrow (1)$ is obvious. 
\end{proof}

\begin{Theorem}\label{3}Consider the data $(\pmb k,d)$ where $\pmb k$ is a field of characteristic $p\ge 3$ and $d$ is a positive integer. 
If $E(4,d\!:\!4)\le \operatorname{mgd}\, \mathrm {Syz}(\pmb k,4,d\!:\!4)$, then  
\begin{equation}\label{cond}\begin{array}{l} \text{$d=kq+r$ for integers $k,q,d$ with $1\le k\le \frac{p-1}2$, $r\in\left\{\frac{q-1}2,\frac{q+1}2\right\}$, and $q=p^e$}\\\text{for some non-negative integer $e$.}\end{array}\end{equation}
\end{Theorem}
\begin{proof} Notice first that $d=\frac{p+1}2$ is one of the numbers described by (\ref{cond}) because $\frac{p+1}2=kq+r$, with $k=\frac{p-1}2$, $q=p^0$, and $r=\frac{q+1}2$; and therefore if $1\le d\le p$, then $d$ is described by (\ref{cond}) if and only if $1\le d\le \frac{p+1}2$. If $d<p$, then the present result follows from Lemma \ref{firstint}. Henceforth, we consider $p\le d$. Write $d$ in the form $d=kq+r$, where $1\le k\le p-1$, $0\le r\le q-1$, and $q=p^e$ for some positive integer $e$.  Part (3) of 
Lemma~\ref{Lemma25.20'}  implies that $E(4,(k+1)\!:\!4)-1\le  \operatorname{mgd}\, \mathrm {Syz}(\pmb k,4,(k+1)\!:\!4)$; and therefore  Lemma~\ref{firstint} implies that 
  $k +1 \le \frac{p+1}{2}$,
$$2k-1=E(4,k\!:\!4)=\operatorname{mgd}\, \mathrm {Syz}(\pmb k,4,k\!:\!4),\ \ \text{and}\ \  2k+1=E(4,(k+1)\!:\!4)=\operatorname{mgd}\, \mathrm {Syz}(\pmb k,4,(k+1)\!:\!4).$$
Parts  (1) and (2) of Lemma~\ref{Lemma25.20'} now give 
$$
2(kq + r) -1= 2d-1 \le \operatorname{mgd}\, \mathrm {Syz}(\pmb k,4,d\!:\!4) \le \mathrm{min} \{ q(2k-1) +4r, q(2k+1)\},
$$
and the conclusion follows.
\end{proof}

\section{The calculation of $\det M_{d,c,c,c}$}

In this section we complete the proof of Theorem~\ref{4}.

\begin{Theorem}\label{4}Fix the data $(\pmb k,d)$ where $\pmb k$ is a  field of characteristic $p\ge 3$ and $d$ is a positive integer. Then 
$$E(4,d\!:\!4)\le \operatorname{mgd}\, \mathrm {Syz}(k,4,d\!:\!4)\iff \text{$d$ is described by {\rm(\ref{cond})}}.$$
In particular, if the field $\pmb k$ is infinite, then $$\text{$A(\pmb k,4,d\!:\!4)$ has the WLP} \iff \text{$d$ is described by {\rm(\ref{cond})}.}$$
\end{Theorem}

\begin{proof} Theorem~\ref{3} establishes the direction $(\Rightarrow)$. The direction $(\Leftarrow)$ is shown in Theorem~\ref{n=4} and Lemma~\ref{.Delta}. The connection between $\operatorname{mgd}\, \mathrm {Syz}(\pmb k,4,d\!:\!4)$ and the WLP may be found in Corollary~\ref{degree} and Remark~\ref{after-degree}. \end{proof}

\begin{Remark}\label{char2} If $\pmb k$ is a field of characteristic $2$ and $d$ is an integer with $2\le d$, then $$\operatorname{mgd}\, \mathrm {Syz}(\pmb k,4,d\!:\!4)<E(4,d\!:\!4),$$ and, if the field $\pmb k$ is infinite, then $A(\pmb k,4,d\!:\!4)$ does not have the WLP; see (\ref{new}). 
\end{Remark}

\begin{Lemma} \label{.Delta} Let $p$  be an odd prime and $\pmb k$ a field of characteristic $p$. 
If $(d,c)$ is a 
 pair of integers  with $1\le c\le d$ and  $d$ of the form 
described in  {\rm(\ref{cond})}, then $\det M_{d,c,c,c}\neq 0$ in $\pmb k$.
\end{Lemma}


\begin{proof}Write $d=kp^e+r$ for integers $k,q,d$ with $1\le k\le \frac{p-1}2$, $r=\frac{p^e+1}2-\epsilon$ where $\epsilon\in\{0,1\}$, and $0\le e$.
 We know from Proposition \ref{32-29-p69}  that 
$$  \det M_{d,c,c,c}= \frac{\binom dc\binom {d+1}c\cdots\binom {d+c-1}c}{\binom cc\binom {c+1}c\cdots\binom {2c-1}c};$$ and therefore, $\det M_{d,c,c,c}$ is equal to 
\begin{equation}\label{.frac}\frac{(d+c-1)^1(d+c-2)^2\cdots(d+1)^{c-1}d^c(d-1)^{c-1}\cdots(d-c+1)^1}{(2c-1)^1(2c-2)^2\cdots(c+1)^{c-1}c^c(c-1)^{c-1}\cdots(1)^1}.\end{equation}
Notice that if $c=d$, then $\det M_{d,c,c,c}$ is  non-zero in $\pmb k$. Henceforth, we assume that $1\le c\le d-1$.

 Let $N$ and $D$ be the    numerator and the denominator of (\ref{.frac}), respectively; and let $o_p(\underline{\phantom{x}})$ be the   
$p$-adic order function, that is, $p^{o_p(N)}$ divides $N$, but $p^{o_p(N)+1}$ does not divide $N$.
 For each positive power $\lambda$, let $N_\lambda=\sum \ell_i$ (respectively,  $D_\lambda=\sum \ell_i$), where the sum is take over all listed factors $v_i^{\ell_i}$ of $N$  (respectively, $D$) from (\ref{.frac}) such that $p^\lambda$ divides $v_i$ in $\Bbb Z$. Observe that 
$$o_p(N)=\sum_{1\le \lambda}N_\lambda\quad \text{and}\quad o_p(D)=\sum_{1\le \lambda}D_\lambda.$$
We show that   $\det M_{d,c,c,c}$ is non-zero in $\pmb k$  by showing $o_p(N)=o_p(D)$ in $\Bbb Z$; indeed, we show that $N_\lambda=D_\lambda$  in $\Bbb Z$, for all positive integers $\lambda$. 

Fix a positive integer $\lambda$ with the property that at least of the integers $N_\lambda$ or $D_\lambda$ is non-zero. In other words,   $p^\lambda$ divides at least one of the listed factors of either $N$ or $D$. Let $\rho_\lambda=\frac{p^\lambda+1}2$. We first observe that  $\lambda\le e$ because 
$$\begin{array}{l}  p^\lambda\le d+c-1\le 2d-2 =2\left(kp^{e}+\frac {p^{e}+1}2-\epsilon\right)-2\\ \phantom{p^\lambda}\le 2 \left (\frac {p-1}2p^{e}+\frac{p^{e}+1}2-\epsilon\right )-2 =p^{e+1}-1-2\epsilon<p^{e+1}.\end{array}$$  Now  we observe that 
\begin{equation}\label{.Impt}d+\epsilon-\rho_\lambda=p^\lambda\cdot u_\lambda,\end{equation}  for some non-negative integer $u_\lambda$. Indeed,
$$d+\epsilon-\rho_\lambda=\left(kp^{e}+\frac {p^{e}+1}2\right)-\frac{p^\lambda+1}2=p^\lambda\left(kp^{e-\lambda}+\frac{p^{e-\lambda}-1}2\right),$$ and we may take $kp^{e-\lambda}+\frac{p^{e-\lambda}-1}2$ to be $u_\lambda$. Finally, we observe that \begin{equation}\label{.rkc}\rho_\lambda\le c.\end{equation} Indeed, we show that if (\ref{.rkc}) fails, then $p^\lambda$ does not divide any of the listed factors of $N$ or $D$. We treat $D$ first. If (\ref{.rkc}) fails, then $2c-1<2\rho_\lambda-1=p^\lambda$ and $p^\lambda$ does not divide any of the listed factors of  $D$. Now we treat $N$. Suppose that (\ref{.rkc}) fails and $\alpha=p^\lambda i$ satisfies $$d-c+1\le \alpha\le d+c-1,$$for some positive integer $i$.
  We know from (\ref{.Impt}) that $d+\epsilon=\rho_\lambda+p^\lambda u_\lambda$; so,
$$\begin{array}{l} p^\lambda u_\lambda=d+\epsilon-\rho_\lambda<d-c+1\le \alpha=p^\lambda i\le d+c-1<d+\rho_\lambda-1\\\phantom{p^\lambda u_\lambda}=(\rho_\lambda+p^\lambda u_\lambda-\epsilon)+\rho_\lambda-1=p^\lambda \rho_\lambda+p^\lambda-\epsilon\le p^\lambda(u_\lambda+1).\end{array}$$ That is,
$p^\lambda u_\lambda<p^\lambda i<p^\lambda(u_\lambda+1)$, which is impossible because the parameters $u_\lambda<i<u_{\lambda}+1$ all are integers. Now that (\ref{.rkc}) is established, we let 
$$ \#_\lambda=\left \lfloor \frac{c-\rho_\lambda}{p^\lambda}\right \rfloor\quad\text{and}\quad  b_\lambda=c-\rho_\lambda- p^\lambda\#_\lambda.$$
Notice that $\#_\lambda$ is a non-negative integer and $0\le b_\lambda\le p^\lambda-1$.  Observe that $D_\lambda=D_\lambda'+D_\lambda''$ with
$$D_\lambda'=\sum\limits_{\left\{\alpha\left\vert 1\le \alpha\le c \text{ and } p^\lambda|\alpha\right.\right\}}\alpha
\qquad\text{and}\qquad D_\lambda''=\sum\limits_{\left\{\alpha\left\vert {c+1\le \alpha\le 2c-1 \text{and } p^\lambda|\alpha}\right.\right\}}(2c-\alpha).$$
We   simplify    $D_\lambda'$. Let $\alpha=i\cdot p^\lambda$, for some positive integer $i$. We add over all $i$ with
$$1\le  i\cdot p^\lambda\le c.$$In other words, $i$ must satisfy:
$$1\cdot p^\lambda\le i\cdot p^\lambda\le  b_\lambda+\rho_\lambda+p^\lambda\#_\lambda.$$
If $b_\lambda+\rho_\lambda<p^\lambda$, then the sum stops at $p^\lambda\#_\lambda$. If $p^\lambda\le b_\lambda+\rho_\lambda$, then the sum stops at $p^\lambda(\#_\lambda+1)$. At any rate,  
$$D_\lambda'=\sum\limits_{i=1}^{\#_\lambda}i\cdot p^\lambda+\chi(p^\lambda\le b_\lambda+\rho_\lambda) p^\lambda(\#_\lambda+1).$$ Our use of ``$\chi$'' is described in (\ref{chi}). The index $i$ in  $D_\lambda''$ must satisfy:
$$1+b_\lambda+\rho_\lambda+p^\lambda\#_\lambda\le i\cdot p^\lambda\le 2(b_\lambda+\rho_\lambda+p^\lambda\#_\lambda)-1=2b_\lambda+(2\#_\lambda+1)p^\lambda$$
If $b_\lambda+\rho_\lambda+1\le p^\lambda$, then the sum starts at $i=\#_\lambda+1$; otherwise, the sum starts at $i=\#_\lambda+2$. The sum always goes at least until $i= 2\#_\lambda+1$. If $p^\lambda\le 2b_\lambda$, then the sum also includes a term for $i= 2\#_\lambda+2$.
Thus,   
$$D_\lambda''=\left\{\begin{array} {l}  \phantom{+}\chi(b_\lambda+\rho_\lambda<p^\lambda)(2c-(\#_\lambda+1)\cdot p^\lambda)+\sum\limits_{i=\#_\lambda+2}^{ 2\#_\lambda+1}(2c-i\cdot p^\lambda)\\+\chi(p^\lambda\le 2b_\lambda)(2c-(2\#_\lambda+2)\cdot p^\lambda).\end{array}\right.$$
 At this point we have
$$D_\lambda=\left\{\begin{array} {l}  \phantom{+}\sum\limits_{i=1}^{\#_\lambda}i\cdot p^\lambda +\chi(p^\lambda\le b_\lambda+\rho_\lambda) p^\lambda(\#_\lambda+1)\\+\chi(b_\lambda+\rho_\lambda<p^\lambda)(2c-(\#_\lambda+1)\cdot p^\lambda)+\sum\limits_{i=\#_\lambda+2}^{ 2\#_\lambda+1}(2c-i\cdot p^\lambda)\\+\chi(p^\lambda\le 2b_\lambda)(2c-(2\#_\lambda+2)\cdot p^\lambda).\end{array}\right.$$ 
Notice that
\begin{equation}\label{t20.1}  \chi(p^\lambda\le b_\lambda+\rho_\lambda)=\chi(p^\lambda\le b_\lambda+{\textstyle\frac{p^\lambda+1}2})=\chi(2\cdot p^\lambda\le 2 b_\lambda+ p^\lambda+1)=\chi( p^\lambda\le 2 b_\lambda +1).\end{equation}Notice also that if $p^\lambda= 2 b_\lambda +1$, then $$\textstyle c=b_\lambda+\rho_\lambda+p^\lambda\#_\lambda =\frac{p^\lambda-1}2+\frac{p^\lambda+1}2+p^\lambda\#_\lambda=(\#_\lambda+1)p^\lambda;$$hence,
$$\chi(p^\lambda= 2b_\lambda+1)(2c-(2\#_\lambda+2)\cdot p^\lambda)=0$$ and 
$$\chi(p^\lambda\le 2b_\lambda)(2c-(2\#_\lambda+2)\cdot p^\lambda)=
\chi(p^\lambda\le 2b_\lambda+1)(2c-(2\#_\lambda+2)\cdot p^\lambda). $$ It follows that 
$$\begin{array}{l}D_\lambda=\left\{\begin{array} {l}  \phantom{+}   \sum\limits_{i=1}^{\#_\lambda}i\cdot p^\lambda +\chi( p^\lambda\le 2 b_\lambda +1) p^\lambda(\#_\lambda+1)\\+\chi(b_\lambda+\rho_\lambda<p^\lambda)(2c-(\#_\lambda+1)\cdot p^\lambda) +\sum\limits_{i=\#_\lambda+2}^{ 2\#_\lambda+1}(2c-i\cdot p^\lambda)\\+(\chi(p^\lambda\le 2b_\lambda+1))(2c-(2\#_\lambda+2)\cdot p^\lambda)\end{array}\right.\\\phantom{D_\lambda}=\left\{\begin{array} {l}  \phantom{+} \sum\limits_{i=1}^{\#_\lambda}i\cdot p^\lambda +\chi(b_\lambda+\rho_\lambda<p^\lambda)(2c-(\#_\lambda+1)\cdot p^\lambda)\\+\sum\limits_{i=\#_\lambda+2}^{ 2\#_\lambda+1}(2c-i\cdot p^\lambda) +\chi(p^\lambda\le 2b_\lambda+1)(2c-(\#_\lambda+1)\cdot p^\lambda).\end{array}\right.\end{array}$$ 
Apply (\ref{t20.1}) again to see that 
 $$\begin{array}{l} D_\lambda=\left\{\begin{array} {l}  \phantom{+} \sum\limits_{i=1}^{\#_\lambda}i\cdot p^\lambda +\chi(b_\lambda+\rho_\lambda<p^\lambda)(2c-(\#_\lambda+1)\cdot p^\lambda)\\+\sum\limits_{i=\#_\lambda+2}^{ 2\#_\lambda+1}(2c-i\cdot p^\lambda)+\chi(p^\lambda\le b_\lambda+\rho_\lambda)(2c-(\#_\lambda+1)\cdot p^\lambda)\end{array}\right.\\ \phantom{D_\lambda}= \sum\limits_{i=1}^{\#_\lambda}i\cdot p^\lambda+\sum\limits_{i=\#_\lambda+1}^{ 2\#_\lambda+1}(2c-i\cdot p^\lambda)\\\phantom{D_\lambda}=2c(\#_\lambda+1)-p^\lambda(\#_\lambda+1)^2.\end{array}$$
Now, we simplify $N_\lambda=N_\lambda'+N_\lambda''$ for $$N_\lambda'=\sum\limits_{\left\{\alpha\left\vert {d-c+1\le \alpha\le d\text{and } p^\lambda|\alpha}\right.\right\}}(c-d+\alpha)\quad\text{and}\quad N_\lambda''=\sum\limits_{\left\{\alpha\left\vert {d+1\le \alpha\le d+c-1\text{and } p^\lambda|\alpha}\right.\right\}}(c+d-\alpha).$$ We continue to write $\alpha=i\cdot p^\lambda$, for some $i$. Recall the integer $u_\lambda$ from (\ref{.Impt}).   Observe that
$$d-c+1\le i\cdot p^\lambda\le d\iff p^\lambda(u_\lambda-\#_\lambda)-b_\lambda-\epsilon+1\le i\cdot p^\lambda \le p^\lambda u_\lambda+\rho_\lambda-\epsilon.$$ In   $N_\lambda'$, the parameter $i$  always stops at $i=u_\lambda$ because $0\le \rho_\lambda-\epsilon<p^\lambda$.
In   $N_\lambda'$, the parameter $i$  begins at $i=u_\lambda-\#_\lambda$; unless $0<-b_\lambda-\epsilon+1$. If $0<-b_\lambda-\epsilon+1$, then  the parameter $i$  does not begin until $i=u_\lambda-\#_\lambda+1$. Thus,  
$$N_\lambda'=-\chi(b_\lambda+\epsilon<1)(c-d+(u_\lambda-\#_\lambda)\cdot p^\lambda) +\sum\limits_{i=u_\lambda-\#_\lambda}^{u_\lambda} (c-d+i\cdot p^\lambda).$$
On the other hand, if  $b_\lambda+\epsilon<1$, then $b_\lambda=\epsilon=0$ and $(c-d+(u_\lambda-\#_\lambda)\cdot p^\lambda)=0$. It follows that
$$ N_\lambda'=\sum\limits_{i=u_\lambda-\#_\lambda}^{u_\lambda} (c-d+i\cdot p^\lambda).$$
We study $N_\lambda''$. When $\alpha=i\cdot p^\lambda$, we use $2\rho_\lambda-1=p^\lambda$ to see that
$$d+1\le \alpha\le d+c-1\iff  p^\lambda\cdot u_\lambda+1+ \rho_\lambda-\epsilon  \le i\cdot p^\lambda\le  p^\lambda\cdot (u_\lambda+\#_\lambda+1)   +b_\lambda -\epsilon.$$ The parameter $i$ in $N_\lambda''$  always begins at $i=u_\lambda+1$ because $0< 1+ \rho_\lambda-\epsilon\le p^\lambda$. 
If $0\le b_\lambda-\epsilon$, then the parameter $i$ in $N_\lambda''$ ends at $u_\lambda+\#_\lambda+1$. If $b_\lambda-\epsilon<0$, then the parameter $i$ ends at $i=u_\lambda+\#_\lambda$. On the other hand, $b_\lambda-\epsilon<0$ only if $b_\lambda=0$ and $\epsilon=1$; and, in this case, the   term from $N_\lambda''$ which corresponds to $i=u_\lambda+\#_\lambda+1$ is $c+d-(u_\lambda+\#_\lambda+1)p^\lambda=0$.  
We conclude that 
$$N_\lambda''=\sum\limits_{i=u_\lambda+1}^{ u_\lambda+\#_\lambda+1}(c+d-ip^\lambda).$$ Thus, 
$$  N_\lambda =\sum\limits_{i=u_\lambda-\#_\lambda}^{u_\lambda} (c-d+i\cdot p^\lambda)+\sum\limits_{i=u_\lambda+1}^{ u_\lambda+\#_\lambda+1}(c+d-ip^\lambda)=2c(\#_\lambda+1)-(\#_\lambda+1)^2p^\lambda.$$ 
We have $N_\lambda=D_\lambda$ and the proof is complete.  \end{proof} 

\section{The WLP for $A(\pmb k,n,d\!:\!n)$ when $5\le n$.}

Our answer to the question ``What is the intuition behind the fact that $A(\pmb k,n,d\!:\!n)$ never has the WLP when $n$ is at least $5$, unless $d$ is very small with respect to the characteristic of $\pmb k$?'',  is contained in the proof of part (1) of Lemma \ref{Lemma25.20}. There are  many ways to produce relations of low degree on
$[x_1^d,\dots,x_{n-1}^d,(x_1+\dots+x_{n-1})^d]$ when $d$ and $n$ are sufficiently large; see especially (\ref{eta}). We nail down the details in Theorem \ref{6.3}. The ultimate result is called Theorem \ref{last}. 

Recall that if $A=\bigoplus_{i\in \mathbb Z} A_i$ is a graded algebra over the field $A_0$, with $A$ finitely generated as an algebra over $A_0$, then the {\it Hilbert function} of $A$ is the function ${\rm H}(A,\underline{\phantom{x}}):\mathbb Z\to \mathbb N$ with ${\rm H}(A,i)$ equal to the dimension of $A_i$ as a vector space over $A_0$. Assertion (1) of Proposition \ref{uni} is a statement about the unimodality of the Hilbert function of a complete intersection. The word ``strictly'' is the key word in the assertion. We have imposed sufficient hypotheses to guarantee that the Hilbert function does not reach a wide plateau before it starts its descent. 
(Notice that $\pmb k[x_1]/(x_1^{a_1})$ has a wide plateau if $2\le a_1$.)
This calculation is well-known by the experts, see for example \cite[Thm. 1]{RRR}. Our interest Hilbert functions is explained in part (2) of Proposition \ref{uni}. Recall, from Corollary \ref{degree}, that in the most natural situation 
$$A(\pmb k,n,\pmb a)\text{ has the WLP}\iff E(n,\pmb a) \le \operatorname{mgd}\, \overline{\mathrm {Syz}}(\pmb k,n,\pmb a).$$ In Proposition \ref{uni} we have identified a hypothesis that guarantees that the reverse inequality automatically holds on the right hand side. Of course, this inequality provides the starting point for the relations of low degree which are built in Lemma \ref{Lemma25.20}.

\begin{Proposition}\label{uni} Fix   $(\pmb k,n,\pmb a)$ as in Data {\rm \ref{xi'}}. 
Let $A=A(\pmb k,n,\pmb a)$.  Assume that $2\le n$ and  $|a_s-a_t|$ is equal to $0$ or $1$  for all indices  $s$ and $t$.
 Then the following statements hold{\rm :}
\begin{itemize}
\item[{\rm(1)}] the Hilbert function ${\rm H}(A,i)$ is a strictly increasing function for $0\le i\le \operatorname{socdeg}(A)/2$, and 
\item[{\rm(2)}] $\operatorname{mgd}\, \overline{\mathrm {Syz}}(\pmb k,n,\pmb a)\le E(n,\pmb a)$.
\end{itemize}\end{Proposition}

\begin{proof} Let $\sigma=\operatorname{socdeg} (A)$. Assume (1) for the time being. We  show that $(1)\Rightarrow(2)$. The $\pmb k$-algebra $A$ is a standard graded Artinian Gorenstein ring; so the Hilbert function of $A$ satisfies the symmetry
\begin{equation}\label{sym}{\rm H}(A,i)={\rm H}(A,\sigma-i),\quad\text{for all integers $i$.}\end{equation} (The symmetry of ${\rm H}(A,\underline{\phantom{x}})$ is well-known; see, for example, \cite[Thm~9.1]{H99}.) We now have 
$${\textstyle {\rm H}(A,\lfloor \frac{\sigma+3}2\rfloor)= {\rm H}(A,\sigma-\lfloor \frac{\sigma-2}2\rfloor)={\rm H}(A,\lfloor \frac{\sigma-2}2\rfloor)< {\rm H}(A,\lfloor \frac{\sigma}2\rfloor)={\rm H}(A,\sigma-\lfloor \frac{\sigma}2\rfloor)={\rm H}(A,\lfloor \frac{\sigma+1}2\rfloor).}$$The inequality in the middle follows from (1); the inner equalities follow from (\ref{sym}) and the outer equalities amount to calculations with rational numbers. At any rate, the vector space $A_{\lfloor \frac{\sigma+3}2\rfloor}$ has less dimension than $A_{\lfloor \frac{\sigma+1}2\rfloor}$ has; multiplication by $L=x_1+\dots+x_{n}$ from $A_{\lfloor \frac{\sigma+1}2\rfloor}$ to $A_{\lfloor \frac{\sigma+3}2\rfloor}$ is not injective; and the minimal generator degree of the kernel of $L:A(-1)\to A$ is at most $\lfloor \frac{\sigma+3}2\rfloor$. (As always, we have $A(-1)_{\lfloor \frac{\sigma+3}2\rfloor}=A_{\lfloor \frac{\sigma+1}2\rfloor}$.) In other words, $${\textstyle \operatorname{mgd}\, K(\pmb k,n,\pmb a)\le  \left\lfloor \frac{\sigma+3}2\right\rfloor.}$$ Apply Theorem \ref{Translation}, (\ref{socdeg}), and Notation \ref{E} to obtain the conclusion of (2).

Now we prove (1). It suffices to prove the result when $2\le a_s$ for all $s$, because, by deleting all indices $s$ with $a_s=1$, one obtains new data $(\pmb k,n',\pmb a')$ with $A(\pmb k,n,\pmb a)=A(\pmb k,n',\pmb a')$ and $2\le a_s'$ for all $s$. It is possible that $n'$ is equal to $0$ or $1$; but in this case, some of the original $a_i$ were equal to $1$ and $A$ is equal to $\pmb k$ or $\pmb k[x_1]/(x_1^2)$. The assertion that ${\rm H}(A,i)$ is a strictly increasing function for $0\le i\le \sigma/2$ is not very interesting for these rings $A$, but it is true.   Henceforth, we assume that $2\le a_s$ for all $s$ and $2\le n$.

Fix an index $i$ with $1\le i\le \sigma/2$. We show that ${\rm H}(A,i-1)<{\rm H}(A,i)$. The proof is by induction on $n$. Assume first that $n=2$. In this case, $\sigma=a_1+a_2-2$, and the hypothesis that $a_2$ differs from $a_1$ by at most $1$ guarantees that $\lfloor\sigma/2\rfloor\le \min\{a_1-1,a_2-1\}$. It follows that 
$${\rm H}(A,i-1)={\rm H}(\pmb k[x_1,x_2],i-1)=i<i+1={\rm H}(\pmb k[x_1,x_2],i)={\rm H}(A,i).$$ Henceforth, we assume that $3\le n$. Partition the monomials of $A_{i-1}$ into two sets  $S_1\cup S_2$, where $S_1$ consists of those monomials that are not divisible by $x_n^{a_n-1}$ and $S_2$ consists of those monomials that are divisible by $x_n^{a_n-1}$. In a similar manner, we partition the monomials of $A_{i}$ in two  sets $T_1\cup T_2$, where $T_1$ consists of those monomials divisible by $x_n$ and $T_2$ consists of those monomials not divisible by $x_n$. We see that 
${\rm H}(A,i-1)=|S_1|+|S_2|$ and ${\rm H}(A,i)=|T_1|+|T_2|$, where $|\text{``set''}|$ is the number of elements of ``set''.
Observe that multiplication by $x_n$ gives a bijection between $S_1$ and $T_1$. To prove the result, it suffices to show that $|S_2|<|T_2|$.

Let $A'=\pmb k[x_1,\dots,x_{n-1}]/
(x_1^{a_1},\dots,x_{n-1}^{a_{n-1}})$ and $\sigma'=\operatorname{socdeg} A'$. We see that $|S_2|={\rm H}(A',i-a_n)$ and $|T_2|={\rm H}(A',i)$. The induction hypothesis applies to $A'$; furthermore, the Hilbert function ${\rm H}(A',\underline{\phantom{x}})$ is symmetric about $\sigma'/2$. To prove the result if suffices to prove
\item[(1)] $i\le \sigma'$, and 
\item[(2)] $|i-\sigma'/2|<|(i-a_n)-\sigma'/2|$.

\noindent  We start with $i\le \sigma/2$; thus,  to show (1), it suffices to show that ${\sigma/2\le \sigma'}$. We see that $\sigma'=\sigma-(a_n-1)$. It suffices to show $\sigma/2\le \sigma-(a_n-1)$; hence, it suffices to show that $2a_n\le \sigma+2$; and this is clear because the hypotheses ensure that $3\le n$ and $a_n\le a_j+1$ for all $j$.  

To prove (2) it is useful to consider the three rational numbers $\lambda<\mu<\nu$ with $\lambda=\sigma'/2$, $\mu=\sigma/2$, and $\nu=\sigma'/2+a_n$. We see that $\mu-\lambda=(a_n-1)/2$ and $\nu-\mu=(a_n+1)/2$; and therefore, \begin{equation}\label{lmn} 0\le \mu-\lambda\le \nu-\mu.\end{equation} The hypothesis gives $i<\mu$. We must prove $|i-\lambda|<|i-\nu|$. The triangle inequality, together with (\ref{lmn}), gives
$$|i-\lambda|=|(i-\mu)+(\mu-\lambda)|\le |i-\mu|+|\mu-\lambda|\le |i-\mu|+(\nu-\mu)=(\mu-i)+(\nu-\mu)=\nu-i=|\nu-i|.$$\end{proof}

\begin{Lemma}\label{Lemma25.20}Consider the data $(\pmb k,n,d)$, where $\pmb k$ is a field of positive characteristic $p$, and $d$ and $n$ are   integers with $1\le d$ and $2\le n$. Write $d=kq+r$ for integers $k,q,r$ with  $0\le r\le q-1$ and $q=p^e$, for some positive integer $e$.
\begin{itemize}
\item[{\rm(1)}]If $1\le k$ and $\ell$ is an integer with $0\le \ell \le n-1$, then 
$$\operatorname{mgd}\, \overline{\mathrm {Syz}}(\pmb k,n,d\!:\!n)\le \left\lfloor\frac{n(k-1)+\ell+3}2\right\rfloor q +r(n-\ell).$$
\item[{\rm(2)}]If $k=1$, then 
$\operatorname{mgd}\, \overline{\mathrm {Syz}}(\pmb k,n,d\!:\!n)\le q+nr$.
\item[{\rm(3)}]If   $\frac q2<d\le q$, then $\operatorname{mgd}\, \overline{\mathrm {Syz}}(\pmb k,n,d\!:\!n)\le q$.
\item[{\rm(4)}]If $k=0$, $p=q$, and $p-d\le (n-1)(d-1)$, then $\operatorname{mgd}\, \overline{\mathrm {Syz}}(\pmb k,n,d\!:\!n)\le p$.
\item[{\rm(5)}]If   $p=3$ and $d=4$, then  $\operatorname{mgd}\, \overline{\mathrm {Syz}}(\pmb k,n,d\!:\!n)\le 9$.
\end{itemize}
\end{Lemma}

\begin{proof}  Let $Q$ be the polynomial ring $\pmb k[x_1,\dots,x_{n-1}]$ and let $D$ represent the data $(\pmb k,n,\pmb a)$ with $\pmb a$ equal to $(d\!:\!n)$. For each assertion, we exhibit a non-zero element of $\overline{\mathrm {Syz}}(D)$ of the appropriate degree.

\medskip\noindent {\bf (1)}\quad Fix integers $k$ and $\ell$ with $1\le k$ and $0\le \ell \le n-1$. 
Assume further that either
\begin{equation}\label{either} \ell=r=0\qquad\text{or}\qquad 1\le r\le q-1\ \text{and}\ 0\le \ell\le n-1.
\end{equation}
Observe that once (1) is established for $\ell=r=0$, then (1) also holds for $0\le \ell\le n-1$ and $r=0$. Indeed, when $r=0$, the minimum of the set $\left\{\left\lfloor\frac{n(k-1)+\ell+3}2\right\rfloor q +r(n-\ell)\mid 0\le \ell \le n-1\right\}$ is $\left\lfloor\frac{n(k-1)+3}2\right\rfloor q$ and this value is attained when $\ell=0$.

Let $D'$ represent the data $(\pmb k,n,\pmb a')$ with $\pmb a'=((k+1)\!:\!\ell,k:(n-\ell))$
  Observe that Proposition \ref{uni} may be applied to the data $D'$. Conclude that 
$$\operatorname{mgd}\, \overline{\mathrm {Syz}}(D')\le \left\lfloor \frac{\ell(k+1)+(n-\ell)k-n+3}2\right\rfloor= \left\lfloor\frac{n(k-1)+\ell+3}2\right\rfloor.$$ Let $\eta'\in \mathrm {Syz}(D')$ be a homogeneous representative of a non-zero element of $\overline{\mathrm {Syz}}(D')$ of degree $\operatorname{mgd}\, \overline{\mathrm {Syz}}(D')$. We have $\eta'=[v_1,\dots,v_n]^{\rm t}\in \mathrm {Syz}(D')$, for some homogeneous polynomials $v_i$ in $\pmb k[x_1,\dots,x_{n-1}]$, with  $\xi(D')\cdot \eta'=0$ and 
\begin{equation}\label{hyp} v_n\notin(x_1^{k+1},\dots,x_{\ell}^{k+1},x_{\ell+1}^k,\dots,x_{n-1}^k)Q.\end{equation}Apply the Frobenius endomorphism to obtain the equation \begin{equation}\label{eq}[\xi(D')]^{[q]}\cdot [\eta']^{[q]}=0\end{equation} with $[\xi(D')]^{[q]}=[x_1^{(k+1)q},\dots,x_{\ell}^{(k+1)q},x_{\ell+1}^{kq},\dots,x_{n-1}^{kq},(x_1+\dots+x_{n-1})^{kq}]$ and
$[\eta']^{[q]}=[v_1^q,\dots,v_n^q]^{\rm t}$. Multiply equation (\ref{eq}) by $x_{\ell+1}^r\cdots x_{n-1}^{r} (x_1+\dots+x_{n-1})^r$ to obtain 
\begin{equation}\label{eta}\eta= \left[\begin{array}{l} 
v_1^qx_1^{q-r}x_{\ell+1}^rx_{\ell+2}^r\cdots x_{n-1}^{r} (x_1+\dots+x_{n-1})^r\\
\phantom{xxxxxxx}\vdots \\
v_\ell^qx_\ell^{q-r}x_{\ell+1}^rx_{\ell+2}^r\cdots x_{n-1}^{r} (x_1+\dots+x_{n-1})^r\\
v_{\ell+1}^q\phantom{x_\ell^{q-r}xx}x_{\ell+2}^r\cdots x_{n-1}^{r} (x_1+\dots+x_{n-1})^r\\
\phantom{xxxxxxx}\vdots \\
v_{n}^q\phantom{x_\ell^{q-r}}x_{\ell+1}^rx_{\ell+2}^r\cdots x_{n-1}^{r} \end{array}\right]\end{equation} in $\mathrm {Syz}(D)$. In other words, $\eta$ is a homogeneous element of $Q(-d)^n$ which is in the kernel of $\xi(D)=[x_1^d,\dots,x_{n-1}^d,(x_1+\dots+x_{n-1})^d]$. It is clear that $$\deg \eta=(\deg \eta')q+r(n-\ell)\le \left\lfloor\frac{n(k-1)+\ell+3}2\right\rfloor q +r(n-\ell).$$ It remains to show that $\eta\notin \mathrm {Kos}(D)$. In other words, it remains   to show that $v_{n}^qx_{\ell+1}^rx_{\ell+2}^r\cdots x_{n-1}^{r}$ is not an element of the ideal  $(x_1^d,\dots,x_{n-1}^d)Q$. The original hypothesis (\ref{hyp}) about $v_n$ guarantees that there is a monomial $m=x_1^{e_1}\dots x_{n-1}^{e_{n-1}}$ which appears in $v_n$ with a non-zero coefficient and for which $$\begin{cases}e_i<k+1&\text{for $1\le i\le \ell$ and}\\e_i<k&\text{for $\ell+1\le i\le n-1$.}\end{cases}$$ Let $M$ be the monomial $m^qx_{\ell+1}^rx_{\ell+2}^r\cdots x_{n-1}^{r}$. We see that $M$ appears in $v_{n}^qx_{\ell+1}^rx_{\ell+2}^r\cdots x_{n-1}^{r}$ with a non-zero coefficient. We must show that  
$M$ is not in the ideal  $(x_1^d,\dots,x_{n-1}^d)Q$. Write
$M=x_1^{\epsilon_1}\cdots x_{n-1}^{\epsilon_{n-1}}$ for
$$\epsilon_i=\begin{cases} e_iq&\text{for $1\le i\le \ell$}\\e_iq+r&\text{for $\ell+1\le i\le n-1$}\end{cases}$$
We see that $\epsilon_i<d$
$$\left\{\begin{array}{lll}  \text{provided $1\le r$}&\text{for $1\le i\le \ell$,}&\text{because $d=kq+r$ and $e_i< k+1$}\\
\text{without any restrictions}&\text{for $\ell+1 \le i\le n-1$,}&\text{because $d=kq+r$ and $e_i< k$}.\end{array}\right.$$The assumptions of (\ref{either}) are in effect; hence,  $\epsilon_i<d$ for all $i$ and the proof of (1) is complete. 

\medskip\noindent {\bf (2)}\quad Proceed as in (1), with $\ell=0$, until reaching $\eta$ as given in (\ref{eta}). Notice that $\eta'=[1,\dots,1,(-1)]$; so $v_n^q=(-1)^q$ and $v_n^qx_1^r\cdots x_{n-1}^r\notin(x_1^d,\dots,x_{n-1}^d)$ because $r<q+r=d$. Thus, $\eta$ represents a non-zero element of $\overline{\mathrm {Syz}}(D)$ of degree $q+nr$.

 \medskip\noindent {\bf (3)}\quad Let $\eta=[x_1^{q-d},\dots,x_{n-1}^{q-d},-(x_1+\dots+x_{n-1})^{q-d}]^{\rm t}$ in $Q(-d)^n$. It is clear that $\eta$ is in the kernel of $\xi(D)$; hence $\eta$ is an element of $\mathrm {Syz}(D)$ of degree $q$. The hypothesis that $\frac q2<d$ guarantees that $q-d<d$ and therefore $\eta\notin \mathrm {Kos}(D)$. 

\medskip\noindent {\bf (4)}\quad Take $\eta$ as described in (3). Notice that the hypothesis $p-d<(n-1)(d-1)$ guarantees that $(x_1+\dots+x_{n-1})^{p-d}\notin(x_1^d,\dots,x_{n-1}^d)$; therefore, $\eta$ represents a non-zero element of $\overline{\mathrm {Syz}}(D)$ of degree $p$. 

\medskip\noindent {\bf (5)}\quad Take $\eta$ as described in (3) with $q=9$. Notice that $x_1^2x_2^3$ appears in $(x_1+\dots+x_{n-1})^5$ with the  non-zero coefficient 1. Thus, $(x_1+\dots+x_{n-1})^{5}\notin(x_1^4,\dots,x_{n-1}^4)$ and 
$\eta$ represents a non-zero element of $\overline{\mathrm {Syz}}(D)$ of degree $9$. 
\end{proof}

\begin{Theorem}\label{6.3} Fix the data $(\pmb k,n,d)$, where $\pmb k$ is a field of positive characteristic $p$, and $d$ and $n$ are positive integers with $p<E(n,d\!:\!n)$ and $5\le n$. Then the following statements hold{\rm:}
\begin{enumerate}
\item $\operatorname{mgd}\, \overline{\mathrm {Syz}}(\pmb k,n,d\!:\!n)< E(n,d\!:\!n)$, and  \item if the field $\pmb k$ is infinite, then 
the ring $A(\pmb k,n,d\!:\!n)$ does not have the WLP.\end{enumerate}
\end{Theorem}

\begin{proof}  In this proof, we write $A$ and $\overline{\mathrm {Syz}}$ in place of $A(\pmb k,n,d\!:\!n)$ and $\overline{\mathrm {Syz}}(\pmb k,n,d\!:\!n)$, respectively. In light of Corollary \ref{degree}, it suffices to prove (1). 
The proof is carried out by analyzing  a large number of cases. In each case, we apply Lemma \ref{Lemma25.20} to estimate an upper bound for $\operatorname{mgd}\, \overline{\mathrm {Syz}}$.

Case 1. Assume that $d=q$, for some $q=p^e$, where $e$ is an integer with $1\le e$. Part (3) of Lemma~\ref{Lemma25.20} gives $\operatorname{mgd}\, \overline{\mathrm {Syz}}\le q$. We see that  $q<\frac{5(q-1)+3}2\le \left\lfloor \frac{n(d-1)+3}2\right\rfloor=E(n,d\!:\!n)$. 

Henceforth in this proof, we may assume that $d$ is not equal to a pure power of $p$.

Case 2. Assume  that $p=2$. The hypothesis $p<E(n,d\!:\!n)$ is equivalent, when $p=2$, to the hypothesis $2\le d$; therefore, we may identify an integer $e$ with $2\le e$ and $2^{e-1}<d< 2^e$. Part (3) of Lemma \ref{Lemma25.20} shows that $\operatorname{mgd}\, \overline{\mathrm {Syz}}\le 2^e$. On the other hand, we see that
\begin{equation}\label{new}2^e=\frac {2^{e+1}}2< \frac {4(2^{e-1})+3}2\le \left\lfloor \frac{n(d-1)+3}2\right\rfloor=E(n,d\!:\!n).\end{equation} In (\ref{new}), we used the inequality $4\le n$; and therefore, we  have shown that $$\operatorname{mgd}\, \overline{\mathrm {Syz}} (\pmb k,n,d\!:\!n)<  E(n,d\!:\!n)$$  when $\pmb k$ is a field of characteristic $2$, $2\le d$, and $n=4$; see Remark \ref{char2}.

Henceforth in this proof, we may assume that $p$ is an odd prime.

Case 3. Assume $d<p$. The hypothesis $p<\lfloor \frac{n(d-1)+3}2\rfloor$ guarantees that $2\le d$; and therefore,
$$p-d<p<\left\lfloor \frac{n(d-1)+3}2\right\rfloor\le   \frac{n(d-1)+3}2 \le \frac{n(d-1)+3(d-1)}2=\frac{n+3}2(d-1)\le(n-1)(d-1). $$Apply part (4) of Lemma \ref{Lemma25.20} to see that $\operatorname{mgd}\, \overline{\mathrm {Syz}}\le p$, which is less than $E(n,d\!:\!n)$ by hypothesis.  

Henceforth in this proof, we may assume that $p<d$. (The case $p=d$ is covered in Case 1.)

Case 4. Assume $p=3$ and $4\le d\le 8$. If $n=5$ and $d=4$, then part (2) of Lemma \ref{Lemma25.20} gives $\operatorname{mgd}\, \overline{\mathrm {Syz}}\le 8$, which is less than $9=E(n,d\!:\!n)$. Throughout the rest of Case 4 we assume that $6\le n$ or $5\le d$. Parts (3) and (5) of Lemma \ref{Lemma25.20} give $\operatorname{mgd}\, \overline{\mathrm {Syz}}\le 9$. If $6\le n$, then $$\textstyle{9<10=\lfloor\frac{6\cdot 3+3}2\rfloor\le \lfloor \frac{n(d-1)+3}2\rfloor=E(n,d\!:\!n).}$$ If $5\le d$, then $9<11=\lfloor\frac{5\cdot 4+3}2\rfloor\le \lfloor \frac{n(d-1)+3}2\rfloor=E(n,d\!:\!n)$.  

Before we consider Case 5, we lay out the plan of attack that will be used in the main body of the argument. As noted in Case 3, we may assume that  $p< d$. Throughout the rest of the proof, we write $d$ in the form    
$$\text{$d=kq+r$ for integers $k,q,r$ with $1\le k$, $0\le r\le q-1$ and $q=p^e$, for some positive integer $e$.}$$ Cases 4 and 2 show that we need only consider $q$ that are at least $5$.   
Part (1) of   Lemma \ref{Lemma25.20} shows that 
$$\operatorname{mgd}\, \overline{\mathrm {Syz}}\le \left\lfloor\frac{n(k-1)+\ell+3}2\right\rfloor q +r(n-\ell),$$ for each integer $\ell$ with $0\le \ell\le n-1$. So, if 
\begin{equation}\label{floor}\left\lfloor\frac{n(k-1)+\ell+3}2\right\rfloor q +r(n-\ell)<\left\lfloor \frac{n(d-1)+3}2\right\rfloor,\end{equation}for some integer $\ell$, with $0\le \ell\le n-1$, then 
$\operatorname{mgd}\, \overline{\mathrm {Syz}}\le E(n,d\!:\!n)$ and the proof is complete for the data $(\pmb k,n,d)$.
 We have $$\frac{n(kq+r-1)+2}2=\frac{n(d-1)+2}2\le \left\lfloor \frac{n(d-1)+3}2\right\rfloor.$$
We also have $$\left\lfloor\frac{n(k-1)+\ell+3}2\right\rfloor
=\frac{n(k-1)+\ell+3-\chi(\text{$n(k-1)+\ell$ is even})}2;$$
 so the inequality (\ref{floor}) is implied by
$$\left(\frac{n(k-1)+\ell+3-\chi(\text{$n(k-1)+\ell$ is even})}2\right)q +r(n-\ell)<\frac{n(kq+r-1)+2}2;$$which is equivalent to 
if \begin{equation}\label{floor2} r (  n  -2\ell )+n-2<(n-\ell-3+\chi(\text{$n(k-1)+\ell$ is even}) )q   .\end{equation}
We have shown that if  (\ref{floor2}) holds  
for some integer $\ell$, with $0\le \ell\le n-1$, then the proof is complete for the data $(\pmb k,n,d)$.

Case 5. Assume $r=0$ and $5\le q$. If $\ell=0$, then \begin{equation}\label{boring}(\ref{floor2})\iff  n-2<(n-3+\chi(\text{$n(k-1)$ is even}) )q.\end{equation} To show the right side of (\ref{boring}), it suffices to observe
$(n-2)<(n-3)5$, and this is clear. Thus, (\ref{floor2}) holds in this case. 

Henceforth in this proof, we may assume that $1\le r$.

Case 6. Assume  $n=5$ and $5\le q$. 
If $\ell=4$, $n=5$, and $k$ is odd, then 
\begin{equation}\label{25.10}{\textstyle (\ref{floor2})\iff \frac 13q+1<r.}\end{equation}
 If $\ell=3$, $n=5$, and $k$ is even,  then
\begin{equation}\label{25.11}(\ref{floor2})\iff   3<r.\end{equation}
If $\ell=2$, $n=5$, and $k$ is odd,  then
\begin{equation}\label{25.12}(\ref{floor2})\iff r<q-3.\end{equation}
If $\ell=1$, $n=5$, and $k$ is even,  then
\begin{equation}\label{25.13}(\ref{floor2})\iff r<{\textstyle \frac 23} q-1.\end{equation}
If $7\le q$ and $k$ is odd, then $\frac 13q+1<q-3$; hence (\ref{25.10}) and (\ref{25.12}) show that   (\ref{floor2}) holds.  
If $q=5$, $k$ is odd,  and $r\neq 2$, then  (\ref{25.10}) and (\ref{25.12}) again show that     (\ref{floor2}) holds.   If $7\le q$ and $k$ is even, then $3<\frac 23 q-1$; hence (\ref{25.11}) and (\ref{25.13})  show that     (\ref{floor2}) holds. If $q=5$, $k$ is even, and $r\neq 3$ then  again (\ref{25.11}) and (\ref{25.13})  show that    (\ref{floor2}) holds. 
It is still necessary to consider $q=5$ and $d$ equal to $1\cdot 5+2$, $3\cdot 5+2$, $2\cdot 5+3$, and $4\cdot 5+3$. If $d=7$, then part (2) of Lemma \ref{Lemma25.20} gives $\operatorname{mgd}\, \overline{\mathrm {Syz}}\le 15<16=E(5,7\!:\!5)$.
 If $d$ is $13$, $17$, or $23$, then $\frac {25}2<d<25$; thus, part (3) of Lemma \ref{Lemma25.20} gives $\operatorname{mgd}\, \overline{\mathrm {Syz}}\le 25<31=\lfloor \frac{5(13-1)+3}2\rfloor\le E(5,d\!:\!5)$. Thus, the inequality of assertion (1) from the statement of Theorem \ref{6.3} holds whenever $n=5$ and $q\le 5$, 

Case 7. Assume  $n=6$ and $5\le q$. 
If $\ell=4$, then $$(\ref{floor2})\iff 2<r.$$ If $\ell=2$, then $$(\ref{floor2})\iff r<q-2.$$ We have $2<q-2$; so, (\ref{floor2}) holds always under the hypotheses of Case 7.

Case 8. Assume $n$ is odd,  $7\le n$, $5\le q$, and $1\le r$. Let $\chi_0=\chi(\text{$k$ is odd})$. If $\ell=3+\chi_0$, then 
$$(\ref{floor2})\iff r(n-6-2\chi_0)+n-2<q(n-6-\chi_0+\chi(\text{$k+\chi_0$ is even})).$$We see that $k+\chi_0$ is always even; therefore $\chi(\text{$k+\chi_0$ is even})=1$ and (\ref{floor2}) is equivalent to 
\begin{equation}\label{25.12'}r(n-6-2\chi_0)+n-2<q(n-5-\chi_0).\end{equation}
If $(n-6-2\chi_0)$ is negative, then $n=7$, $\chi_0=1$, and (\ref{25.12'}) holds. Otherwise,  $0\le(n-6-2\chi_0)$ 
  and $$r(n-6-2\chi_0)\le (q-1)(n-6-2\chi_0),$$since  $r\le q-1$. To prove (\ref{25.12'}), it suffices to 
prove $$(q-1)(n-6-2\chi_0)+n-2<(n-5-\chi_0)q$$ and this equivalent to 
  $4+2\chi_0<(1+\chi_0)q$. The most recent inequality holds because $5\le q$ and $\chi_0$ is either $0$ or $1$. 
So, (\ref{floor2}) holds always under the hypotheses of Case 8.

Case 9. Assume  $n$ is even,  $8\le n$, and $5\le q$. If $\ell=\frac{n}2$, then 
\begin{equation}\label{n=8}{\textstyle (\ref{floor2})\iff n-2<(\frac n2-3+\chi(\text{$\frac n2$ is even}))q.}\end{equation}
The inequality on the right side of (\ref{n=8}) holds when $n=8$ because $6<(2)5\le 2q$. For $10\le n$, it suffices to observe that $n-2<(\frac  n2-3)5$; and this is clear.
  \end{proof}

\begin{Theorem} \label{last}Fix the data $(\pmb k,n,d)$, where $\pmb k$ is a field of positive characteristic $p$, and $d$ and $n$ are positive integers with $5\le n$. Then the following statements hold{\rm:}
\begin{enumerate}
\item $E(n,d\!:\!n)\le   \operatorname{mgd}\,\overline{\mathrm {Syz}}(\pmb k,n,d\!:\!n)
\iff E(n,d\!:\!n)\le p$, 
and  \item if the field $\pmb k$ is infinite, then 
the ring $A(\pmb k,n,d\!:\!n)$ has the WLP if and only if $\left\lfloor\frac{n(d-1)+3}2\right\rfloor\le p$.\end{enumerate}
\end{Theorem}
\begin{proof}The integer $E(n,d\!:\!n)$ is equal to $\left\lfloor\frac{n(d-1)+3}2\right\rfloor$; thus, 
in light of Corollary \ref{degree} it suffices to establish (1). The direction 
$$ \operatorname{mgd}\,\overline{\mathrm {Syz}}(\pmb k,n,d\!:\!n)< E(n,d\!:\!n)\Leftarrow p<E(n,d\!:\!n)$$ is Theorem \ref{6.3}. The direction 
$$ \text{inequality (\ref{Stan}) for $(\pmb k,n-1,d\!:\!(n-1),d)$}\Leftarrow E(n,d\!:\!n)\le p$$ is Theorem \ref {5<=n}. Of course,
$$\text{inequality (\ref{Stan}) for $(\pmb k,n-1,d\!:\!(n-1),d)$}\iff E(n,d\!:\!n)\le   \operatorname{mgd}\,\overline{\mathrm {Syz}}(\pmb k,n,d\!:\!n)$$is (5)$\iff$(4) in Corollary \ref{degree}.
\end{proof}

\end{document}